\documentclass[11pt,a4paper]{amsart}
\usepackage[T1]{fontenc}
\usepackage{hyperref}
\usepackage{amssymb}

\theoremstyle{plain}
\newtheorem{thm}{Theorem}[section]
\newtheorem{lemma}[thm]{Lemma}
\newtheorem{corollary}[thm]{Corollary}
\theoremstyle{definition}
\newtheorem{defi}[thm]{Definition}
\newtheorem{prop}[thm]{Proposition}
\theoremstyle{remark}
\newtheorem{remark}[thm]{Remark}

\begin{document}
\title{Speedups for Presburger Arithmetic and Real Closed Fields}

\author{Julien Daoud$^1$}
\author{Fedor Pakhomov$^{1, 2}$}

\thanks{$^1$Department of Mathematics: Analysis, Logic and Discrete Mathematics (WE16), Ghent University, Krijgslaan 297, S8, 9000 Ghent, Belgium. Emails: \href{mailto:julien.daoud@ugent.be}{\texttt{julien.daoud@ugent.be}}, \href{mailto:fedor.pakhomov@ugent.be}{\texttt{fedor.pakhomov@ugent.be}}}
\thanks{$^2$Steklov Mathematical Institute of Russian Academy of Sciences, 8 Gubkina St., Moscow 119991, Russia.}
\thanks{Funding: Work of both authors was supported by the BOF grant BOF.STG.2022.0042.01. Work of Fedor Pakhomov was additionally supported by the FWO grant G0F8421N\@.}


\begin{abstract}
In the present paper, we consider Presburger arithmetic $\mathsf{PrA}$ and the theory of real closed fields $\mathsf{RCF}$. Due to quantifier elimination in these theories, there are two kinds of natural ways to axiomatize them. Namely, on one hand, $\mathsf{PrA}$ can be axiomatized with the full schema of first-order induction, and $\mathsf{RCF}$ with the full schema of the first-order least upper bound principle. At the same time, there are natural axiomatizations of these theories that avoid the use of formulas of unbounded quantifier depth. In the present paper, we compare these two groups of axiomatizations from the perspective of proof lengths. We show that the first group of axiomatizations enjoys at least a double exponential speedup. For Presburger arithmetic we also show, by formalizing a quantifier elimination procedure, that every true sentence has a proof of at most triple exponential size in the second axiomatization, so that the speedup is at most triple exponential in this case.
\end{abstract}

\maketitle

\section{Introduction}
The investigation of the speedup phenomenon for first-order theories, i.e. situations where, for some sentences provable in both theories $T$ and $U$, their proofs in $T$ are much shorter than in $U$, goes back to the classical result of G\"odel \cite{Goed36}, who discovered super-recursive speedups for situations where $T$ is stronger than $U$ (see also \cite{Ehr71}). The situation, however, becomes substantially different when dealing with pairs of conservative theories. Notable results here are the speedup results of Pudl\'ak \cite{Pud86} for von Neumann-G\"odel-Bernays set theory $\mathsf{NGB}$ over $\mathsf{ZF}$ and for $\mathsf{ACA}_0$ over $\mathsf{PA}$, where the speedup is more than elementary, its magnitude being that of the hyper-exponential function $\mathsf{exp}^*(0)=1$, $\mathsf{exp}^*(x+1)=2^{\mathsf{exp}^*(x)}$. The technique used there is based on the use of finite consistency statements and cut-shortening, naturally leading to this magnitude of speedup (see e.g. \cite{Kol24}). For a survey on the study of proof lengths, see Pudl\'ak's survey \cite{Pud98}.

In this paper, we look at the speedup phenomenon in the domain of complete decidable theories and present a case study on real closed fields and Presburger arithmetic. The decidability of Presburger arithmetic is a classical result going back to the late 1920s \cite{Pre29}. The decidability of the theory of real closed fields is due to Tarski \cite{Tar51}. For both theories, decidability is achieved via quantifier elimination when the theories are formulated in suitable signatures.

In the present paper, we consider two styles of axiomatizations for these two theories. One group consists of axiomatizations that are the schematic counterparts of the natural second-order axiomatizations of the standard models of the respective theories. Another group consists of more minimalistic axiomatizations, isolating the principles necessary to formalize quantifier elimination. We establish a double-exponential speedup of the schematic axiomatizations over the corresponding more explicit axiomatizations. The precise notion of an $f$-speedup that we use is fixed in Section~\ref{sec:conventions}; all our lower bounds are of the form $f(x)=2^{2^{x^{\varepsilon}}}$ for some constant $\varepsilon>0$, where $x$ measures the size of the shortest proof in the schematic axiomatization.

For Presburger arithmetic, a more efficient axiomatization $\mathsf{PrA}$ takes the finitely axiomatized theory $\mathsf{PrA}^-$ (which provides basic properties of the intended standard model $(\mathbb{N};=,0,1,+)$) and extends it with the full schema of induction. Another axiomatization of the set of all true first-order sentences in the model $(\mathbb{N};=,0,1,+)$ that we consider is $\mathsf{PrA}_{\mathit{alt}}$, defined as the extension of $\mathsf{PrA}^-$ by modulo-comparison axioms indexed by prime numbers $p$, stating that every $x$ is congruent modulo $p$ to some number in $\{0,\ldots,p-1\}$ (the uniqueness of this number is provable already in $\mathsf{PrA}^-$) (see \cite{Gra97,Vis11}). Our central result regarding $\mathsf{PrA}$ is that we define a sequence of sentences $(\varphi_i\in \mathcal{L}_{\mathsf{PrA}})_{i<\omega}$, where $\varphi_i$ has a length polynomial in $i$ and a proof in $\mathsf{PrA}$ polynomial in $i$, but for some constant $\varepsilon>0$, its shortest proof in $\mathsf{PrA}_{\mathit{alt}}$ has a length $\ge 2^{2^{i^{\varepsilon}}}$ (Theorem~\ref{thm:speedupPrA}). Concretely, $\varphi_i$ is the sentence $\forall x\,\mathit{Div}_i(x)$ expressing that every number is divisible with remainder by every $0<y\le 2^{2^i}$, where the bound $2^{2^{i}}$ is expressed by a formula of size polynomial in $i$. Any $\mathsf{PrA}_{\mathit{alt}}$-proof of $\varphi_i$ is forced to use a modulo-comparison axiom for a prime $p>2^{2^{i}-1}$, and such an axiom is itself of size $>2^{2^i-1}$.

For real closed fields, we first consider the axiomatization $\mathsf{RCF}$, which consists of the axioms of ordered fields, formulated with the order defined by $x\le y:\equiv\exists z\; x+z\cdot z=y$ (so that every non-negative element is a square), together with the existence of square roots for exactly the non-negative numbers, and the existence of roots for odd-degree polynomials (in this formulation the square-root axiom is in fact redundant, see Remark~\ref{rem:sqrt}; we keep it for recognizability). The second, more powerful axiomatization $\mathsf{Tarski}$ that we consider extends the same base theory with the least upper bound schema. For the purposes of investigating proof length, this corresponds to Tarski's axiomatization of elementary geometry \cite{Tar99}, with the continuity schema there corresponding to the least upper bound schema in our case. Analogously to the previous case, we prove that $\mathsf{Tarski}$ has at least a double-exponential speedup over $\mathsf{RCF}$ (Theorem~\ref{thm:speedupRCF}). Here the witnessing sentences are the sentences $\mathit{Root}_n$ expressing that $2$ has a $y$-th root for every natural number $1<y\le 2^{2^n}$; any $\mathsf{RCF}$-proof of $\mathit{Root}_n$ is forced to use an odd-degree root axiom of degree $>2^{2^n-1}$.

On a technical level, the main idea of our proofs is to construct short sentences expressing the validity of much longer axioms from the less efficient axiomatizations in such a way that the produced sentences have short proofs in the more efficient axiomatization. These techniques are rather similar to those employed by Fischer and Rabin in their classical 1974 paper, which gives lower bounds on the computational complexity of the decision procedures for these theories \cite{FischerRabin1974}.

Additionally, we establish triple exponential upper bounds for the lengths of proofs of true sentences in $\mathsf{PrA}_{\mathit{alt}}$, thus also providing a triple exponential upper bound on the speedup of $\mathsf{PrA}$ over $\mathsf{PrA}_{\mathit{alt}}$ (Theorem~\ref{thm:triple} and Corollary~\ref{cor:upper}). We conjecture that an analogous phenomenon, with a finite tower of exponentiations as an upper bound on the proofs of true sentences, holds for the theory $\mathsf{RCF}$; we expect that this might be possible to achieve via a formalization of quantifier elimination for real closed fields via cylindrical algebraic decompositions \cite{collins1975quantifier,basu2006algorithms}.

\section{Preliminaries}
\subsection{Conventions}\label{sec:conventions}
Throughout the paper we use the following conventions.

\smallskip\noindent\emph{Sizes.} Formulas and proofs are strings over a fixed finite alphabet; a variable is written as the letter $x$ followed by its index in binary, so that the $i$-th variable is a string of length $O(\log i)$. The \emph{size} (or \emph{length}) $|\alpha|$ of a formula or proof $\alpha$ is the number of symbols of the corresponding string. All the results of the paper are robust under this choice: the bounds ``polynomial in $n$'' and ``$\ge 2^{2^{x^{\varepsilon}}}$ for some constant $\varepsilon>0$'' are invariant under replacing $|\cdot|$ by any polynomially related measure.

\smallskip\noindent\emph{Numerals and multiples.} For a natural number $n$ we write $\bar n$, or simply $n$ when no confusion is possible, for the term $1+(\ldots+(1+1)\ldots)$ with $n$ occurrences of $1$ ($\bar 0$ is the constant $0$); its size is $\Theta(n)$. For a term $t$ and a natural number $k$ we write $kt$ (or $k\cdot t$, $t\cdot k$, $tk$) for the $k$-fold sum $t+(\ldots+(t+t)\ldots)$, a term of size $\Theta(k\cdot|t|)$. Thus multiplication by standard natural numbers is available as an abbreviation in every language containing $+$.

\smallskip\noindent\emph{Proof system.} We fix a standard Hilbert-style calculus $\mathsf{H}$ for first-order logic with equality (finitely many axiom schemas; the rules modus ponens and generalization). A proof of $\varphi$ in a theory $T$ is a finite sequence of formulas ending with $\varphi$, each of which is a logical axiom, an axiom of $T$, or follows from earlier formulas of the sequence by a rule. Two consequences of this choice will be used repeatedly:
\begin{enumerate}
\item every axiom of $T$ used in a proof occurs as a line of the proof; hence, if every $T$-proof of $\varphi$ must use some axiom of size $\ge N$, then every $T$-proof of $\varphi$ has size $\ge N$;
\item Hilbert-style calculi, natural deduction and sequent calculi with cut simulate each other with polynomial overhead \cite[p.~554]{Pud98}. If a calculus $\mathsf{S}$ is simulated by $\mathsf{H}$ with overhead $m\mapsto Cm^{d}$, then a polynomial upper bound on $\mathsf{S}$-proofs stays polynomial for $\mathsf{H}$-proofs, and a lower bound $2^{2^{n}-1}$ on $\mathsf{H}$-proofs yields the lower bound $(2^{2^{n}-1}/C)^{1/d}\ge 2^{2^{n-d'}}$ (for a suitable constant $d'$ and all large $n$) on $\mathsf{S}$-proofs, which is again of the form $2^{2^{n^{\varepsilon}}}$. Hence all speedup statements below hold for any proof system of this class, up to the value of the constant $\varepsilon$.
\end{enumerate}

\smallskip\noindent\emph{Schemas.} An axiom schema applies to all formulas $\phi(x,\vec v)$ of the language in question, where $\vec v$ is a (possibly empty) tuple of parameter variables; every axiom is understood as the universal closure of the displayed formula.

\smallskip\noindent\emph{Speedup.} We say that a theory $T$ has an \emph{$f$-speedup} over a theory $U$ in a common language $\mathcal{L}$ (cf.\ \cite{Pud98}) if there is a sequence of $\mathcal{L}$-sentences $(\varphi_n)_{n<\omega}$, each provable in both $T$ and $U$, such that, writing $m_n$ for the minimal size of a $T$-proof of $\varphi_n$, we have $m_n\to\infty$ and, for all sufficiently large $n$, every $U$-proof of $\varphi_n$ has size at least $f(m_n)$. Since a proof contains the sentence it proves, $m_n\ge|\varphi_n|$; hence, for non-decreasing $f$, an $f$-speedup in this sense also yields the lower bound $f(|\varphi_n|)$ on $U$-proofs in terms of the size of the sentence. All the speedups established in this paper are of the form $f(x)=2^{2^{x^{\varepsilon}}}$ for some constant $\varepsilon>0$.

\subsection{Short iterated definitions}
As is common in investigations of proof lengths, we will use Solovay's classical result (see also \cite{FR79}) regarding short formulas expressing iterated schematic definitions.
\begin{thm}[Solovay, see {\cite[p.~557]{Pud98}}]\label{eff_def}
Let $\mathcal{L}$ be a first-order language with equality containing the constants $0$ and $1$ and, besides $\neg$, at least one of the connectives $\rightarrow$, $\vee$, $\wedge$. Let $R$ be a $k$-ary relation symbol not in $\mathcal{L}$, let $\vec a=a_1,\ldots,a_k$ and $\vec b$ be disjoint tuples of variables, let $\phi_0(\vec a,\vec b)$ be an $\mathcal{L}$-formula and let $\Phi(R,\vec a,\vec b)$ be an $\mathcal{L}\cup\{R\}$-formula, both with free variables among $\vec a,\vec b$ (the variables $\vec b$ play the role of parameters). For an $\mathcal{L}$-formula $\psi(\vec a,\vec b)$ we denote by $\Phi(\psi,\vec a,\vec b)$ the result of replacing in $\Phi$ every atomic subformula $R(\vec t\,)$ by $\psi(\vec t,\vec b)$ (renaming bound variables so as to avoid capture).

Then there are $\mathcal{L}$-formulas $\phi_1(\vec a,\vec b),\phi_2(\vec a,\vec b),\ldots$ with free variables among $\vec a,\vec b$, and a polynomial $q$ depending only on $\phi_0$ and $\Phi$, such that for every $n\ge 0$ the sentence
\[\forall\vec a\,\forall\vec b\,\big(\phi_{n+1}(\vec a,\vec b)\leftrightarrow\Phi(\phi_n,\vec a,\vec b)\big)\]
has a proof of size at most $q(n)$ in the first-order predicate calculus with equality $\mathsf{H}$ from the single non-logical axiom $0\neq1$. In particular, $|\phi_n|\le q(n)$ for every $n$.
\end{thm}

The key idea of the proof of this theorem is to replace the formula $\Phi$ with an equivalent formula where $R$ appears only once, thereby avoiding an exponential blow-up.

This result enables one to define arbitrary formulas iteratively but take them to be equivalent to some polynomial-size (in the number of iterations) formulas. Note that the size bound on $\phi_n$ is a consequence of the bound on the proofs: the formula $\phi_n$ occurs as a subformula of the last line of the proof of the $n$-th equivalence, so $|\phi_n|\le q(n)$. Note also that the theories $\mathsf{PrA}^-$ and $\mathsf{RCF}^-$ introduced below prove $0\ne1$ (by axiom 2 of $\mathsf{PrA}^-$, respectively by the field axioms), so the equivalences of Theorem~\ref{eff_def} are available, with polynomial-size proofs, inside all the theories considered in this paper.

This is the approach we will take. We shall define bounded multiplication for Presburger arithmetic, i.e., multiplication will be defined for one of the two factors bounded by a certain quantity (here $2^{2^n}$) depending on the number $n$ of iterations of the formula $\Phi$. Similarly, we define bounded exponentiation for the theory of real closed fields.

By controlling the meaning of these formulas as well as their size (the number of characters), we will be able to define a certain sequence of formulas that will have short (i.e., polynomial) proofs in our first axiomatization of Presburger arithmetic but will only be provable with long axioms, and hence with long proofs, in the second axiomatization. The same idea will apply to the axiomatizations for real closed fields.

\section{Speedup in Presburger Arithmetic}\label{sec:PrA}
\subsection{Definitions}
\begin{defi}
Let $\mathcal{L}_{\mathsf{PrA}}$ be the language of Presburger arithmetic, which is the language of first-order logic with equality, constants $0$ and $1$, and the binary function symbol $+$.
\end{defi}
\begin{defi}
Let the theory $\mathsf{PrA}^-$ be the theory with the following axioms:\\\\
1. Axioms of cancellative Abelian semigroup with neutral element
$0$ for $+$\\
2. $\forall x \;  x+1\neq0$\\
3. $\forall x \;  x\neq0 \; \rightarrow \exists y \; x=y+1$\\
4. $\forall x,y \; x\leq y \vee y \leq x,$\\
where $x \leq y$ is a shorthand for the formula $\exists z \; x+z=y$, and $x<y$ is a shorthand for the formula $\exists z\; x+z+1=y$ (equivalently, in $\mathsf{PrA}^-$, $x\le y\wedge x\ne y$).
\end{defi}
Let us recall a standard fact of the model theory of arithmetic.
\begin{prop}\label{lem:initial}
Let $M$ be a model of $\mathsf{PrA}^-$. Then the map $n\mapsto \bar n^{M}$ is an embedding of $(\mathbb{N};0,1,+,\le)$ onto an initial segment of $M$: the elements $\bar n^{M}$ are pairwise distinct, and for every $a\in M$ and every $n\in\mathbb{N}$, if $M\models a\le\bar n$, then $a=\bar k^{M}$ for some $k\le n$.
\end{prop}

We thus have that all models of $\mathsf{PrA}^-$ contain $(\mathbb{N};=,0,1,+)$ as their initial segments. Therefore, we will freely identify the standard naturals with the elements of the said initial segments of the models of $\mathsf{PrA}^{-}$. For each standard natural number $n$, of course, it is the value of the term \[\displaystyle \underbrace{1+(\ldots +(1+1)\ldots)}\limits_{\mbox{$n$-times}}\] in all models of $\mathsf{PrA}^{-}$, and hence we can freely use the standard numbers within $\mathsf{PrA}^{-}$-formulas. Within models of $\mathsf{PrA}^{-}$, we define multiplication by a standard natural $n$ as a function mapping a given \[x\mapsto\displaystyle \underbrace{x+(\ldots +(x+x)\ldots)}\limits_{\mbox{$n$-times}}.\]

\begin{lemma}[On bounded size definability of bounded multiplication]\label{lem:Mul}
In $\mathcal{L}_{\mathsf{PrA}}$, we can define formulas $\mathit{Mul}_n (x,y,z)$ ($n$ being a natural number), which express in all models of $\mathsf{PrA}^{-}$ that $y$ is a standard natural number with $y\leq 2^{2^n}$ and $x\cdot y=z$. These formulas are of a size polynomial in $n$.
\end{lemma}
\begin{proof}
Let $\mathit{Mul}_0 (x,y,z)$ be: $$ (y=0\rightarrow z=0) \wedge (y=1\rightarrow z=x) \wedge (y=2\rightarrow z=x+x) \wedge \neg (y\neq 0 \wedge y \neq 1 \wedge y \neq 2)$$
It is straightforward to check that the formula works as expected.

For $n>0$, we define $\mathit{Mul}_{n}(x,y,z)$ to be equivalent to:
\begin{equation}\label{mul_n_eq}\begin{aligned} \exists y_1,y_2,y_3,y_4,z_1,z_2,z_4&(y=y_3+y_4\land \mathit{Mul}_{n-1}(y_1,y_2,y_3)\land \mathit{Mul}_{n-1}(x,y_1,z_1)\land\\ &  \mathit{Mul}_{n-1}(z_1,y_2,z_2)\land \mathit{Mul}_{n-1}(x,y_4,z_4)\land\\ &z=z_2+z_4 \land y_2<y_1).\end{aligned}\end{equation}
With the help of Theorem \ref{eff_def} (applied to $\phi_0=\mathit{Mul}_0$ and to the formula $\Phi(R,x,y,z)$ obtained from (\ref{mul_n_eq}) by replacing $\mathit{Mul}_{n-1}$ with a fresh ternary relation symbol $R$), we construct $\mathit{Mul}_n$ to be of a size polynomial in $n$, such that $\mathit{Mul}_n(x,y,z)$ is equivalent to (\ref{mul_n_eq}) with a proof polynomial in $n$.

To see that $\mathit{Mul}_n$ works in the expected way, assuming that $\mathit{Mul}_{n-1}$ works in the intended way, we observe that the matrix of (\ref{mul_n_eq}) entails that:
\begin{enumerate}
    \item \label{bound_mul_def_cond1}$y_3=y_1\cdot y_2$ and $y=y_1\cdot y_2+y_4$ with $y_1,y_2,y_4\le 2^{2^{n-1}}$ standard and $y_2<y_1$
    \item $z_1=xy_1$, $z_2=(x y_1)y_2$, $z_4=xy_4$, and $z=(x y_1)y_2+xy_4=x(y_1y_2+y_4)$.
\end{enumerate}
Note that condition (\ref{bound_mul_def_cond1}) entails that $y$ has to be a standard natural number less than or equal to $2^{2^n}$, since the maximal possible value of $y$ represented in this way is attained with $y_1=2^{2^{n-1}}$, $y_2=2^{2^{n-1}}-1$, and $y_4=2^{2^{n-1}}$. Furthermore, every other $y< 2^{2^n}$ has a representation as in (\ref{bound_mul_def_cond1}) with $y_1=2^{2^{n-1}}$ and $y_2$ and $y_4$ being the result and the remainder of the division of $y$ by $y_1$.
\end{proof}
\begin{defi}

We define $\mathit{Hyp}^{\mathit{Mul}}_n(y)$ to be the conjunction of the following formulas:
\begin{enumerate}
    \item $\mathit{Mul}_{n}(1,y,y)$,
    \item $\mathit{Mul}_n(0,y,0)$,
    \item $\forall x \exists ! z \; \mathit{Mul}_n(x,y,z)$,
    \item \[\begin{aligned}\forall x_0, x_1, x_2 \; \Bigg(x_0=x_1+&x_2 \to \\ &  \exists z_0,z_1,z_2 \Bigg (\begin{aligned}\mathit{Mul}_{n}(x_1,y,z_1)\wedge \mathit{Mul}_{n}(x_2,y,z_2)\\\wedge \mathit{Mul}_{n}(x_0,y,z_0)\wedge z_0=z_1+z_2\end{aligned}\Bigg )\Bigg).\end{aligned}\]
\end{enumerate}
\end{defi}

(The superscript $\mathit{Mul}$ distinguishes this formula from the formula $\mathit{Hyp}^{\mathit{Pow}}_n$ playing the analogous role in Section~\ref{sec:RCF}.)

\begin{lemma}\label{lem:HypMul}
Let $M$ be a model of $\mathsf{PrA}^-$ and $y\in M$. Then $M\models \mathit{Hyp}^{\mathit{Mul}}_n(y)$ if and only if $y$ is a standard natural number $\le 2^{2^n}$.
\end{lemma}
\begin{proof}
If $M\models\mathit{Hyp}^{\mathit{Mul}}_n(y)$, then in particular $M\models \mathit{Mul}_n(1,y,y)$, and by Lemma~\ref{lem:Mul} $y$ is a standard natural number $\le 2^{2^n}$. Conversely, let $y=\bar m$ with $m\le 2^{2^n}$. By Lemma~\ref{lem:Mul}, for all $x,z\in M$ we have $M\models\mathit{Mul}_n(x,y,z)$ if and only if $z=x\cdot m$ is the $m$-fold sum of $x$. This function is total and single-valued, which gives conjunct (3); $1\cdot m=\bar m$ and $0\cdot m=0$ give conjuncts (1) and (2); and $(x_1+x_2)\cdot m=x_1\cdot m+x_2\cdot m$ holds in every commutative semigroup, which gives conjunct (4).
\end{proof}

We say that in a model of $\mathsf{PrA}^{-}$, an element $a$ is divisible by a standard natural $n$ with a remainder if, for some element of the model $b$ and some $0\le r<n$, we have $a=bn+r$.

\begin{defi}
    We define $\mathit{Div}_n(x)$ as the formula:
    \[\forall y ( \mathit{Hyp}^{\mathit{Mul}}_n (y)\rightarrow (y=0) \vee \exists a_1, b_1, b_2(\mathit{Mul}_{n}(a_1,y,b_1)\wedge x=b_1+b_2\wedge b_2<y)).\]
\end{defi}
Clearly, in all models of $\mathsf{PrA}^-$, the formulas $\mathit{Div}_n(x)$ express that an element $x$ is divisible with a remainder by all standard numbers $0<y\le 2^{2^n}$ (by Lemma~\ref{lem:HypMul} and Proposition~\ref{lem:initial}, which shows that $b_2<y$ forces $b_2\in\{0,\ldots,y-1\}$). The formulas $\mathit{Div}_n(x)$ have size polynomial in $n$, since $\mathit{Mul}_n$ does.

\subsection{Theories}
\begin{defi}
Let $\mathsf{PrA}$ be the theory $\mathsf{PrA}^-$ with the schema of induction.\\\\
i.e. For all formulas $\phi(x,\vec v)$ with at least one free variable $x$ (and possibly parameters $\vec v$), we have the axiom (the universal closure of) $$\phi (0,\vec v) \wedge \forall x (\phi(x,\vec v)\rightarrow \phi(x+1,\vec v)) \rightarrow \forall x \,\phi(x,\vec v)$$
\end{defi}
\begin{defi}
Let $\mathsf{PrA}_{\mathit{alt}}$ be the theory $\mathsf{PrA}^-$ with the following axioms.\\\\
For each prime $p$, we have the axiom $$\forall x\;(x\equiv_p 0 \vee ... \vee x\equiv_p p-1),$$
where $x\equiv_n s$ is a shorthand for the formula $$\exists z \; (\underbrace{z+\ldots+z}\limits_{\mbox{$n$-times}}+s=x \vee \; \underbrace{z+\ldots+z}\limits_{\mbox{$n$-times}}+x=s).$$
\end{defi}
\begin{remark}\label{rem:PrAalt}
The axiom for $p$ has $p$ disjuncts, each of size $\Theta(p)$ (a $p$-fold sum together with a numeral $s<p$), so its size is $\Theta(p^2)$. The axioms only assert the \emph{existence} of a residue; the uniqueness of the residue $s\in\{0,\ldots,p-1\}$ with $x\equiv_p s$ is provable in $\mathsf{PrA}^-$ (see Lemma~\ref{lem:modD}). It is well known that $\mathsf{PrA}_{\mathit{alt}}$ is a complete axiomatization of the theory of $(\mathbb{N};=,0,1,+)$ \cite{Gra97,Vis11}; this also follows from Theorem~\ref{thm:triple} below (Section~\ref{sec:upper} does not depend on the present section). Both $\mathsf{PrA}$ and $\mathsf{PrA}_{\mathit{alt}}$ are sound for the standard model, so they axiomatize the same complete theory.
\end{remark}
\subsection{Speedup}

\begin{lemma}\label{lem:PrApoly}
In $\mathsf{PrA}$, there is a proof of $\forall x\, \mathit{Div}_n (x)$ of size polynomial in $n$.
\end{lemma}
\begin{proof} The idea of our proof for the lemma is to form a single schema of a proof in $\mathsf{PrA}$. Namely, the proofs $P_n$ of $\forall x \mathit{Div}_n (x)$ will be such that there is a single proof template $P(R^{(3)})$ built from formulas of the language of $\mathsf{PrA}$ extended by a single ternary predicate letter $R^{(3)}$, and for each $n$, when we replace $R^{(3)}$ with $\mathit{Mul}_n$, we obtain precisely $P_n$. This implies that the lengths of $P_n$ depend linearly on the lengths of $\mathit{Mul}_n$, and hence the lengths of $P_n$ are polynomial in $n$. We stress which properties of $\mathit{Mul}_n$ the template consumes: none. In $P(R^{(3)})$ the symbol $R^{(3)}$ is treated as an uninterpreted predicate; it occurs only inside the induction axiom for the formula $\mathit{Div}_n(x)$ (with $R^{(3)}$ in place of $\mathit{Mul}_n$) and inside the conjuncts of $\mathit{Hyp}^{\mathit{Mul}}_n(y)$ that are used as hypotheses. In particular, the recursive definition (\ref{mul_n_eq}) of $\mathit{Mul}_n$ is never unfolded. All the auxiliary facts about $+$, $\le$ and $<$ used below are provable in $\mathsf{PrA}^-$ by proofs of a fixed size independent of $n$. The same pattern will be used again in Section~\ref{sec:RCF}.

The proof $P_n$ in $\mathsf{PrA}$ is a proof of $\mathit{Div}_n (x)$ by induction on $x$. Recall that the definition of $\mathit{Div}_n(x)$ is: $$\forall y ( \mathit{Hyp}^{\mathit{Mul}}_n (y)\rightarrow (y=0) \vee \exists a_1, b_1, b_2(\mathit{Mul}_{n}(a_1,y,b_1)\wedge x=b_1+b_2\wedge b_2<y)).$$

First, let us prove $\mathit{Div}_n (x)$ for the base case of $x=0$.
We take $a_1$, $b_1$, and $b_2$ to all be equal to $0$. Now, since clearly \[\forall y ( \mathit{Hyp}^{\mathit{Mul}}_n (y)\rightarrow (y=0) \vee (\mathit{Mul}_{n}(0,y,0)\wedge 0=0+0\wedge 0<y)),\]
(conjunct (2) of $\mathit{Hyp}^{\mathit{Mul}}_n(y)$ gives $\mathit{Mul}_n(0,y,0)$, and $0<y$ follows from $y\neq0$ by axiom 3),
we conclude $\mathit{Div}_n (0)$.

Now we verify the induction step. That is, we assume $\mathit{Div}_n (x)$ holds and show that $\mathit{Div}_n (x+1)$ holds.

To prove $\mathit{Div}_n (x+1)$, we assume $\mathit{Hyp}^{\mathit{Mul}}_n(y)$ for some $y\neq 0$ and claim that there are $a_1$, $b_1$, and $b_2$ such that $\mathit{Mul}_n(a_1,y,b_1)$, $x+1=b_1+b_2$, and $b_2<y$. From the induction hypothesis, it follows that there are $a_1'$, $b_1'$, and $b_2'$ such that $\mathit{Mul}_n(a_1',y,b_1')$, $x=b_1'+b_2'$, and $b_2'<y$.

Since $y\neq0$, axiom 3 of $\mathsf{PrA}^-$ gives some $w$ with $y=w+1$. Unfolding $b_2'<y$, we get $u$ with $b_2'+u+1=w+1$, hence $b_2'+u=w$ by cancellation, and we distinguish the cases $u\ne0$ (then $u=u'+1$ and $b_2'<w$) and $u=0$ (then $b_2'=w$).

If $b_2'<w$, then $b_2'+1<w+1=y$, and we clearly satisfy the target conditions by taking $a_1=a_1'$, $b_1=b_1'$, and  $b_2=b_2'+1$.

If $b_2'=w$, then we set $a_1=a_1'+1$, $b_1=b_1'+y$, and $b_2=0$. We use $\mathit{Hyp}^{\mathit{Mul}}_n(y)$ to verify that this triple $a_1,b_1,b_2$ is as required. Conjunct (4) of $\mathit{Hyp}^{\mathit{Mul}}_n(y)$, applied to $x_0=a_1'+1$, $x_1=a_1'$, $x_2=1$, gives $z_0,z_1,z_2$ with $\mathit{Mul}_n(a_1',y,z_1)$, $\mathit{Mul}_n(1,y,z_2)$, $\mathit{Mul}_n(a_1'+1,y,z_0)$ and $z_0=z_1+z_2$. By conjunct (1) and the uniqueness part of conjunct (3), $z_2=y$ and $z_1=b_1'$; hence $z_0=b_1'+y=b_1$ and $\mathit{Mul}_n(a_1,y,b_1)$. Moreover, $x+1=(b_1'+w)+1=b_1'+(w+1)=b_1'+y=b_1+0$, and $0<y$ since $0+w+1=y$. This completes the induction step, and the induction axiom for $\mathit{Div}_n(x)$ yields $\forall x\,\mathit{Div}_n(x)$.
\end{proof}

For the following lemma, we use the construction of models of $\mathsf{PrA}^{-}$ from {\cite[p.~329]{Smo91}}.
\begin{lemma}\label{lem:PrAlower}
Let $n\ge1$. Every proof of $\forall x\,\mathit{Div}_n (x)$ in $\mathsf{PrA}_{\mathit{alt}}$ must use at least one axiom whose size is greater than $2^{2^{n}-1}$. Consequently, every proof of $\forall x\,\mathit{Div}_n(x)$ in $\mathsf{PrA}_{\mathit{alt}}$ has size greater than $2^{2^{n}-1}$.
\end{lemma}
\begin{proof}
The strategy of this proof is to construct, for every prime $p$, a model $M_p$ that satisfies all the axioms of $\mathsf{PrA}^-$ and the axioms $\forall x \; (x\equiv_m 0 \vee \ldots \vee x\equiv_m m-1)$ for all primes $m<p$, but fails the sentence $\forall x\,\mathit{Div}_n(x)$ whenever $p\le 2^{2^n}$. The construction itself does not depend on $n$; the bound $p\le 2^{2^n}$ is only needed to guarantee $M_p\models\mathit{Hyp}^{\mathit{Mul}}_n(p)$.

Let $p$ be a prime and let $p_0,p_1,\ldots,p_t$ be the sequence of all primes strictly smaller than $p$.
Let $M_p$ be the model whose elements are the naturals $\mathbb{N}$ \emph{and} all polynomials of the form
$$a_1\cdot X+a_0,$$
where $a_0$ is an integer and $a_1$ is a positive rational of the form
$$\frac{q}{p_0^{r_0}\cdot p_1^{r_1}\cdots p_t^{r_t}},$$
with $q$ a positive natural and all $r_i$ natural numbers. Addition is that of the polynomial ring $\mathbb{Q}[X]$; the constants $0$ and $1$ are interpreted as themselves. Note that the set of admissible leading coefficients $a_1$ is closed under addition and under division by each $p_i$, and that $M_p$ is closed under addition.

We check that $M_p\models\mathsf{PrA}^-$. It is a sub-semigroup with $0$ of the cancellative Abelian group $(\mathbb{Q}[X],+)$, which gives axiom 1. Axiom 2 holds since $x+1$ is never $0$ (for a polynomial, $x+1$ is a polynomial). For axiom 3, a non-zero natural has a predecessor in $\mathbb{N}$, and $a_1X+a_0=(a_1X+(a_0-1))+1$ with $a_1X+(a_0-1)\in M_p$. For axiom 4, note that $x\le y$ holds in $M_p$ iff $y-x\in M_p$, and that every polynomial of $M_p$ lies above every natural: given $x,y\in M_p$, either both are naturals (and they are comparable), or exactly one is a polynomial (which is then the larger one), or $x=a_1X+a_0$ and $y=a_1'X+a_0'$; if $a_1<a_1'$ then $y-x\in M_p$, if $a_1'<a_1$ then $x-y\in M_p$, and if $a_1=a_1'$ then $y-x$ or $x-y$ is a natural number.

Next, $M_p$ satisfies the axiom $\forall x\,(x\equiv_m 0\vee\ldots\vee x\equiv_m m-1)$ for every prime $m<p$, i.e., $m=p_i$ for some $i\le t$. For $x\in\mathbb{N}$ this is ordinary division with remainder. For $x=a_1X+a_0$, let $s\in\{0,\ldots,m-1\}$ be the residue of $a_0$ modulo $m$ and put $c:=\frac{a_1}{m}X+\frac{a_0-s}{m}$; then $c\in M_p$ (the denominator of $\frac{a_1}{m}$ is again of the admissible form) and $m\cdot c+s=x$, so $x\equiv_m s$.

Now assume $p\le 2^{2^n}$. Then $M_p\models\mathit{Hyp}^{\mathit{Mul}}_n(p)$ by Lemma~\ref{lem:HypMul}, and clearly $p\ne0$. We show that $\mathit{Div}_n(X)$ fails in $M_p$, i.e., that the matrix
\[\exists a_1,b_1,b_2\,(\mathit{Mul}_n(a_1,p,b_1)\wedge X=b_1+b_2\wedge b_2<p)\]
is false in $M_p$. By Lemma~\ref{lem:Mul}, $\mathit{Mul}_n(a_1,p,b_1)$ forces $b_1$ to be the $p$-fold sum $p\cdot a_1$. By Proposition~\ref{lem:initial}, and since every polynomial of $M_p$ lies above every natural, $b_2<p$ forces $b_2\in\{0,\ldots,p-1\}$. Thus we would have $p\cdot a_1=X-b_2$ in $\mathbb{Q}[X]$. The element $a_1$ cannot be a natural number (then $p\cdot a_1$ would be a natural), so $a_1$ is a polynomial of $M_p$ with leading coefficient $\frac1p$; but $\frac1p$ is not an admissible leading coefficient, since $p\notin\{p_0,\ldots,p_t\}$. Hence no such $a_1,b_1,b_2$ exist, and $M_p\not\models\forall x\,\mathit{Div}_n(x)$.

Finally, by Bertrand's postulate (proved by Chebyshev \cite{Che52}), for every natural $m>1$ there is a prime $k$ with $m<k<2m$; since $n\ge1$, we may apply this to $m=2^{2^n-1}\ge2$ and obtain a prime $k$ with $2^{2^{n}-1}<k<2^{2^{n}}$. Consider the model $M_k$. Let $\pi$ be a proof of $\forall x\,\mathit{Div}_n(x)$ in $\mathsf{PrA}_{\mathit{alt}}$. If all the axioms of $\mathsf{PrA}_{\mathit{alt}}$ occurring in $\pi$ were true in $M_k$, then by soundness $M_k\models\forall x\,\mathit{Div}_n(x)$, which we have just refuted. Hence $\pi$ uses some axiom that is false in $M_k$. All axioms of $\mathsf{PrA}^-$ and all modulo-comparison axioms for primes $m<k$ are true in $M_k$, so the axiom in question is a modulo-comparison axiom for a prime $q\ge k$ (the model $M_k$ may fail such axioms for primes $q>k$ as well, but this only helps: any such axiom is at least as long as the one for $k$). By Remark~\ref{rem:PrAalt}, the size of the axiom for $q$ is $\Theta(q^2)$; in particular it is greater than $q\ge k>2^{2^{n}-1}$. Since axioms occur as lines of proofs (Section~\ref{sec:conventions}), the size of $\pi$ is also greater than $2^{2^n-1}$.
\end{proof}
\begin{remark}
The bound of Lemma~\ref{lem:PrAlower} is not tight: the argument actually shows that a proof must use an axiom of size $\Theta(k^2)$ with $k>2^{2^{n}-1}$, i.e., of size $\Omega(2^{2^{n+1}-2})$. We state the weaker bound because it is what the speedup theorem below needs.
\end{remark}
\begin{thm}[Speedup]\label{thm:speedupPrA}
There is a constant $\varepsilon>0$ such that $\mathsf{PrA}$ has a $2^{2^{x^{\varepsilon}}}$-speedup over $\mathsf{PrA}_{\mathit{alt}}$, witnessed by the sentences $\varphi_n:=\forall x\,\mathit{Div}_n(x)$. More precisely: $\varphi_n$ has size polynomial in $n$, has a $\mathsf{PrA}$-proof of size polynomial in $n$, and, for $n\ge1$, every $\mathsf{PrA}_{\mathit{alt}}$-proof of $\varphi_n$ has size greater than $2^{2^{n}-1}$.
\end{thm}
\begin{proof}
First, the sentences $\varphi_n$ are provable in both theories. Provability in $\mathsf{PrA}$ is Lemma~\ref{lem:PrApoly}. Provability in $\mathsf{PrA}_{\mathit{alt}}$ follows from the completeness of $\mathsf{PrA}_{\mathit{alt}}$ (Remark~\ref{rem:PrAalt}), since $\varphi_n$ is true in $(\mathbb{N};=,0,1,+)$: by Lemma~\ref{lem:HypMul}, any $y$ with $\mathit{Hyp}^{\mathit{Mul}}_n(y)$ is a natural number $\le2^{2^n}$, and every natural number is divisible with remainder by every positive natural number.

Let $m_n$ be the minimal size of a $\mathsf{PrA}$-proof of $\varphi_n$. By Lemma~\ref{lem:PrApoly} there are constants $C,c\ge1$ with $m_n\le Cn^{c}$ for all $n$, and $m_n\ge|\varphi_n|\to\infty$ (by Lemma~\ref{lem:PrAlower}, the minimal size of a $\mathsf{PrA}_{\mathit{alt}}$-proof of $\varphi_n$ is unbounded in $n$, so each sentence occurs only finitely often in the sequence $(\varphi_n)_{n<\omega}$, and there are only finitely many sentences of any given size). Put $\varepsilon:=1/(2c)$. Then $m_n^{\varepsilon}\le C^{\varepsilon}\sqrt{n}\le n-1$ for all sufficiently large $n$, and therefore, by Lemma~\ref{lem:PrAlower}, every $\mathsf{PrA}_{\mathit{alt}}$-proof of $\varphi_n$ has size greater than
\[2^{2^{n}-1}\ \ge\ 2^{2^{n-1}}\ \ge\ 2^{2^{m_n^{\varepsilon}}}\]
for all sufficiently large $n$. This is exactly the definition of a $2^{2^{x^{\varepsilon}}}$-speedup given in Section~\ref{sec:conventions}. By item (2) of the conventions on proof systems, the same conclusion holds (with a possibly smaller $\varepsilon$) for any proof system polynomially equivalent to $\mathsf{H}$.
\end{proof}

\section{Speedup over the Theory of Real Closed Fields}\label{sec:RCF}
In this section we follow the same five-step scheme as in Section~\ref{sec:PrA}: (i) a short iterated definition, now $\mathit{Pow}_n(x,y,z)$ expressing $x^y=z$ for natural exponents $y\le 2^{2^n}$, in place of $\mathit{Mul}_n$; (ii) a formula $\mathit{Hyp}^{\mathit{Pow}}_n(y)$ collecting the properties of $\mathit{Pow}_n$ that the short proof needs, in place of $\mathit{Hyp}^{\mathit{Mul}}_n$; (iii) a target sentence $\mathit{Root}_n$, ``$2$ has a $y$-th root for every $1<y\le 2^{2^n}$'', in place of ``every $x$ is divisible with remainder by every $0<y\le2^{2^n}$''; (iv) a polynomial-size proof of the target sentence from the schematic axiomatization, now via the least upper bound schema instead of induction; and (v) a model-theoretic lower bound showing that any proof in the minimalistic axiomatization must use a specific very long axiom, now an odd-degree root axiom instead of a modulo-comparison axiom.
\begin{defi}
Let $\mathcal{L}_{\mathsf{RCF}}$ be the language of the theory of real closed fields in first-order logic, consisting of the constants $0$ and $1$, and the function symbols $+$ and $\cdot$.
\end{defi}

\begin{defi}
Let $\mathsf{RCF}^-$ be the theory consisting of the following axioms in $\mathcal{L}_{\mathsf{RCF}}$:
\begin{enumerate}
    \item  Axioms of fields.
    \item Axioms of linear order for the relation $\leq$, compatible with the field operations, where $x\leq y$ is a shorthand for $\exists z \; x+z\cdot z = y$. Explicitly: reflexivity, antisymmetry and transitivity of $\le$, totality $\forall x,y\,(x\le y\vee y\le x)$, and compatibility with the field operations $\forall x,y,z\,(x\le y\rightarrow x+z\le y+z)$ and $\forall x,y\,(0\le x\wedge 0\le y\rightarrow 0\le x\cdot y)$.
\end{enumerate}
We write $x<y$ for $x\le y\wedge x\ne y$, and $x\ge y$, $x>y$ for $y\le x$, $y<x$. We use the shorthand $x^2$ for $x\cdot x$ and, more generally, $x^i$ for the $i$-fold product $x\cdot x\cdots x$ ($x^0=1$).
\end{defi}
\begin{remark}\label{rem:Euclidean}
In every model $M$ of $\mathsf{RCF}^-$, the map $m\mapsto\bar m$ is an embedding of $(\mathbb{N};0,1,+,\cdot,\le)$ into $M$, where $\le$ is the defined relation. Indeed, $\bar a+\bar b=\overline{a+b}$ and $\bar a\cdot\bar b=\overline{ab}$ hold by the field axioms; $\bar a+1\cdot1=\overline{a+1}$ gives $\bar a\le\overline{a+1}$, and $\bar a\ne\overline{a+1}$ because $1\ne0$; hence $\bar0<\bar1<\bar2<\cdots$, and by transitivity and antisymmetry of $\le$ we get that $\bar a\le\bar b$ if and only if $a\le b$, and that the $\bar m$ are pairwise distinct. We call the elements $\bar m$ of $M$ the \emph{standard natural numbers} of $M$ and identify them with the corresponding natural numbers.
\end{remark}
\begin{lemma}\label{sqr}
Let $M$ be a subfield of $\mathbb{R}$ with the induced operations. Then $M\models\mathsf{RCF}^-$ if and only if every non-negative element of $M$ has a square root in $M$; moreover, in that case, for $x,y\in M$ the defined relation $x\le y$ holds in $M$ if and only if $x\le y$ in $\mathbb{R}$.
\end{lemma}
\begin{proof}
Assume first that every non-negative element of $M$ has a square root in $M$, and let $x,y\in M$. If $x+z\cdot z=y$ for some $z\in M$, then $y-x=z\cdot z\ge0$ in $\mathbb{R}$. If, on the other hand, $x\le y$ in $\mathbb{R}$, then $y-x$ is a non-negative element of $M$, hence by our assumption $y-x=z\cdot z$ for some $z\in M$, i.e.\ $x\le y$ holds in $M$ in the defined sense. So the defined relation coincides on $M$ with the usual order of $\mathbb{R}$, and therefore the axioms of $\mathsf{RCF}^-$, which hold in $\mathbb{R}$ with its usual order, hold in $M$. Assume now that $M\models\mathsf{RCF}^-$, and let $z\in M$ with $z\ge0$ in $\mathbb{R}$. By totality, $0\le z$ or $z\le0$ in the defined sense, i.e.\ $z=w\cdot w$ or $z=-w\cdot w$ for some $w\in M$; in the second case $z\le0$ in $\mathbb{R}$, so $z=0=0\cdot0$. In either case $z$ has a square root in $M$, which also means that the second equivalence holds by the argument above.
\end{proof}

\begin{lemma}[On bounded size definability of bounded exponentiation]\label{lem:Pow}
In $\mathcal{L}_{\mathsf{RCF}}$, we can define the formulas $\mathit{Pow}_n (x,y,z)$ ($n$ being a natural number), which express in all models of $\mathsf{RCF}^-$ that $y$ is a standard natural number, $y\leq 2^{2^n}$, and $x^y=z$. These formulas are of polynomial size in $n$.
\end{lemma}
\begin{proof}
We use the notation of Solovay's theorem. Let $\phi_0(\vec{a}, \vec{b} )$ be $\mathit{Pow}_0 (x,y,z)$. We define the formula $\mathit{Pow}_0 (x,y,z)$ as:
$$(y=0\rightarrow z=1) \wedge (y=1\rightarrow z=x) \wedge (y=2\rightarrow z=x\cdot x) \wedge \neg (y\neq 0 \wedge y\neq 1 \wedge y\neq 2)$$
Here, $2$ is a shorthand for $1+1$ in $\mathcal{L}_{\mathsf{RCF}}$. The final conjunct ensures that $y$ is a natural number up to $2$. Thus, this formula functions as intended.

Let $R$ be a fresh ternary relation symbol. We define $\Phi(R,x,y,z)$ to be the formula
$$
\begin{aligned}
\exists y_1 \Big(&R(1,y_1,1)\wedge\forall y_2 (R(1,y_2,1)\rightarrow y_2 \leq y_1)\wedge y\leq y_1\cdot y_1\Big)\\
&\wedge (R(1,y,1)\rightarrow R(x,y,z))\\
&\wedge \Big(\neg R(1,y,1)\rightarrow \exists y_3, y_4, y_5, z_2, z_3, z_4 \big( \\
&\quad \: R(1,y_3,1)\wedge R(1,y_4,1)\wedge R(1,y_5,1) \\
&\quad \wedge y=y_3\cdot y_4+y_5 \wedge R(x,y_3,z_2) \wedge R(z_2,y_4,z_3) \\
&\quad \wedge R(x,y_5,z_4)\wedge z=z_3\cdot z_4 \big)\Big)
\end{aligned}
$$
and let $\mathit{Pow}_n:=\phi_n$ be the formulas given by Theorem~\ref{eff_def} for this $\phi_0$ and $\Phi$; thus $\mathit{Pow}_n(x,y,z)$ is provably equivalent to $\Phi(\mathit{Pow}_{n-1},x,y,z)$ with a proof of size polynomial in $n$, and $|\mathit{Pow}_n|$ is polynomial in $n$.

The matrix of this formula can be understood as follows:
\begin{enumerate}
    \item The first part states that $y_1$ is the largest element satisfying $R(1,y_1,1)$; when $R$ is $\mathit{Pow}_{n-1}$, this element is $2^{2^{n-1}}$ by the induction hypothesis, so that this part states $y\leq 2^{2^{n-1}}\cdot 2^{2^{n-1}}=2^{2^n}$.
    \item The second part dictates that if $y\leq 2^{2^{n-1}}$, then $\mathit{Pow}_n(x,y,z)$ holds if and only if $\mathit{Pow}_{n-1}(x,y,z)$ holds.
    \item The final part explains how to evaluate $x^y$ when $y>2^{2^{n-1}}$. This is achieved by finding integer elements such that $y=y_3\cdot y_4+y_5$ with $y_3,y_4,y_5\leq 2^{2^{n-1}}$ (which exist because $y\leq (2^{2^{n-1}})^2$), allowing $\mathit{Pow}_{n-1}$ to be iterated. This relies on the identity $x^y = x^{y_3\cdot y_4+y_5} = (x^{y_3})^{y_4} \cdot x^{y_5} = z$, where all constituent exponentiations are well-defined within $\mathit{Pow}_{n-1}$.
\end{enumerate}

We thus have that if $M\models\mathsf{RCF}^-$ and $x,y,z\in M$, then $M\models\mathit{Pow}_n(x,y,z)$ if and only if $y=\bar m$ for some natural number $m\le2^{2^n}$ and $z=x^m$, where $\bar m$ denotes the value of the numeral $\bar m$ in $M$, and $x^m$ the $m$-fold product of $x$ in $M$. Indeed, by passing between $\mathbb{N}$ and $M$ through Remark~\ref{rem:Euclidean} (which gives us $\overline{a+b}=\bar a+\bar b$, $\overline{ab}=\bar a\cdot\bar b$, and $\bar a\le\bar b$ iff $a\le b$), we argue by induction on (the level $n$ of) $\mathit{Pow}_{n}$ to obtain the equivalence.
\end{proof}

\begin{defi}
We define $\mathit{Hyp}^{\mathit{Pow}}_n(y)$ to be the conjunction of the following formulas:
\begin{enumerate}
    \item[(0)] $\forall x>1\;\exists z\;\mathit{Pow}_n(x,y,z)$
    \item $\exists x,z \; (1<x \wedge \mathit{Pow}_n(x,y,z)\wedge z\leq 2)$
    \item $\forall x>1\;\forall x_1>1 \;\forall z, z_1 \; \Big((\mathit{Pow}_n (x,y,z)\wedge \mathit{Pow}_n(x_1,y,z_1)) \rightarrow z>x \wedge (z_1>z \rightarrow x_1>x)\Big)$
    \item $\forall c_1, c_2>1 \; \Big(c_1\neq c_2 \rightarrow \exists x_0,c_0 \; (\mathit{Pow}_n(x_0,y,c_0)\wedge (c_1<c_0<c_2 \vee c_2<c_0<c_1) \wedge x_0>1)\Big)$
\end{enumerate}
\end{defi}

The conjuncts of $\mathit{Hyp}^{\mathit{Pow}}_n(y)$ guarantee the necessary analytic behavior: totality of $x\mapsto x^y$ on $(1,\infty)$, proximity to 1, strict monotonicity, and density of $y$-th powers, respectively.

\begin{defi}
We define $\mathit{Root}_n$ as the formula:
$$\forall y (\mathit{Hyp}^{\mathit{Pow}}_n(y) \rightarrow \neg(y>1)\vee \exists r \mathit{Pow}_{n}(r,y,2))$$
\end{defi}

This formula asserts that, under the analytic hypotheses, one can find an $r$ such that $2=r^y$ if $y>1$. Note that, in contrast to Lemma~\ref{lem:HypMul}, we do not claim that $\mathit{Hyp}^{\mathit{Pow}}_n(y)$ holds for all standard $1<y\le2^{2^n}$ in every model of $\mathsf{RCF}^-$; we only need this for the subfields of $\mathbb{R}$ considered in Lemma~\ref{RCFmodelsHyp} below. Because $\mathit{Pow}_n$ is of polynomial size in $n$, the resulting formulas $\mathit{Hyp}^{\mathit{Pow}}_n(y)$ and $\mathit{Root}_n$ are clearly of polynomial size as well.

\begin{defi}
Let $\mathsf{RCF}$ be the theory $\mathsf{RCF}^-$ with the following axioms:

\begin{enumerate}
   \item  $\forall z\,(0\le z \rightarrow \exists x \; x\cdot x=z)$
\item  For each odd natural number $m$, we include the axiom:
$$\forall a_0, ... , a_{m-1} \exists x \; x^{m} + a_{m-1}x^{m-1} + \ldots + a_0 = 0 $$
where $x^i$ is a shorthand for $x \cdot x \cdots x$ ($i$ times).
\end{enumerate}
\end{defi}
\begin{remark}\label{rem:sqrt}
With the defined order, the hypothesis $0\le z$ of axiom (1) says that $z=w\cdot w$ for some $w$, so axiom (1) holds in every field; we keep it for recognizability only. 
Also notice that the axiom of odd degree $m$ contains the $m$-fold product $x^m$, so its size is greater than $m$.
\end{remark}

\begin{defi}
Let $\mathsf{Tarski}$ be the theory $\mathsf{RCF}^-$ augmented with the least upper bound schema of axioms (LUB).

For all formulas $\phi(x,\vec v)$ with at least one free variable $x$ (and possibly parameters $\vec v$, which we suppress in the notation), we have the axiom (or the universal closure of):
\begin{multline*}
\Big(\exists d\, \phi (d)\wedge \exists b\,\forall x\, \big(\phi (x) \rightarrow x\leq b\big)\Big) \rightarrow\\ \exists c\, \Big(\forall x\, \big(\phi (x) \rightarrow x\leq c\big)\wedge\forall b\, \big(\forall x\, (\phi (x) \rightarrow x\leq b)\rightarrow c\leq b\big)\Big)
\end{multline*}
\end{defi}

\subsection{Speedup}
\begin{lemma}\label{lem:Tarskipoly}
In $\mathsf{Tarski}$, there is a proof of $\mathit{Root}_n$ of size polynomial in $n$.
\end{lemma}
\begin{proof}
Recall that the formula $\mathit{Root}_n$ is defined as $$\forall y (\mathit{Hyp}^{\mathit{Pow}}_n(y) \rightarrow \neg(y>1)\vee \exists r \mathit{Pow}_{n}(r,y,2))$$
We argue inside $\mathsf{Tarski}$. Consider an arbitrary $y$ that satisfies the hypotheses $\mathit{Hyp}^{\mathit{Pow}}_n(y)$ and $y>1$; we must prove that there exists an $r$ such that $\mathit{Pow}_{n}(r,y,2)$. We invoke the LUB property on the formula $\phi(x,y)$ defined as:
$$\phi(x,y) := \exists z \;\big ( 1< x \wedge \mathit{Pow}_{n}(x,y,z)\wedge z\leq 2\big )$$
(with parameter $y$). By the proximity to $1$ hypothesis (conjunct (1)), $\exists d \; \phi (d,y)$ holds. Furthermore, the strict monotonicity hypothesis (conjunct (2), applied with $x_1=x$ and $z_1=z$) ensures that $b=2$ functions as an upper bound (i.e., $\forall x \; (\phi (x) \rightarrow x\leq 2)$), because if $\phi(x)$ holds via $z$, then $x<z\le2$. Thus, the least upper bound $c$ exists; note that $c\ge d>1$, so $c>1$, which we use below since conjuncts (2) and (3) only speak about arguments $>1$.

We now demonstrate that $c$ is the required root. By the totality hypothesis (conjunct (0)) there is $c_1$ with $\mathit{Pow}_n(c,y,c_1)$, and by conjunct (2), $c_1>c>1$. Assume for contradiction that $c_1 \neq 2$. By the density hypothesis (conjunct (3), applied to $c_1>1$ and $c_2=2>1$), there exists an $x_0 > 1$ and $c_0$ with $\mathit{Pow}_n(x_0,y,c_0)$ such that $c_0$ lies strictly between $c_1$ and $2$. This produces a contradiction in both possible cases:
\begin{enumerate}
    \item \textbf{Case $c_1< 2$:} Then $c_1<c_0<2$. By strict monotonicity (conjunct (2) applied to the pairs $(c,c_1)$ and $(x_0,c_0)$), $x_0>c$; but $\phi(x_0)$ holds (witnessed by $c_0\le2$), which means $c$ cannot be an upper bound.
    \item \textbf{Case $c_1> 2$:} Then $2<c_0<c_1$, so by conjunct (2) applied to $(x_0,c_0)$ and $(c,c_1)$ we get $x_0<c$. Moreover, $x_0$ acts as an upper bound for all $x$ satisfying $\phi(x)$: if $\phi(x)$ holds via $z$, then $z\le 2<c_0$, and conjunct (2) applied to $(x,z)$ and $(x_0,c_0)$ gives $x<x_0$. Consequently, $c$ is not the \emph{least} upper bound.
\end{enumerate}
Hence $c_1=2$, i.e., $\mathit{Pow}_n(c,y,2)$. This concludes the proof of $\mathit{Root}_n$.

Concerning the size of the proof, we argue exactly as in Lemma~\ref{lem:PrApoly}. The derivation above treats $\mathit{Pow}_n$ as a black box: it uses only the conjuncts of $\mathit{Hyp}^{\mathit{Pow}}_n(y)$, one instance of the LUB schema whose formula $\phi$ contains $\mathit{Pow}_n$ as a subformula, and elementary facts about $\le$ and $<$ provable in $\mathsf{RCF}^-$ by proofs of fixed size; the recursive definition of $\mathit{Pow}_n$ is never unfolded. Hence there is a single proof template $P(R^{(3)})$, built from formulas of $\mathcal{L}_{\mathsf{RCF}}$ extended by a ternary predicate letter $R^{(3)}$, such that substituting $\mathit{Pow}_n$ for $R^{(3)}$ yields a $\mathsf{Tarski}$-proof of $\mathit{Root}_n$ of size linear in $|\mathit{Pow}_n|$, hence polynomial in $n$ by Lemma~\ref{lem:Pow}.
\end{proof}

\begin{lemma}\label{RCFmodelsHyp}
Let $M$ be a subfield of $\mathbb{R}$ such that every non-negative element of $M$ has a square root in $M$ (equivalently, by Lemma~\ref{sqr}, $M$ with the induced operations is a model of $\mathsf{RCF}^-$; note that $\mathbb{Q}\subseteq M$ automatically). Then, for every natural number $y$ with $1<y\leq 2^{2^{n}}$, $M$ satisfies $\mathit{Hyp}^{\mathit{Pow}}_n(y)$.
\end{lemma}
\begin{proof}
Since $M\models\mathsf{RCF}^-$, Lemma~\ref{lem:Pow} applies in $M$: for $x,z\in M$ and our fixed natural number $y$, $M\models\mathit{Pow}_n(x,y,z)$ if and only if $z=x^y$; moreover, the defined order of $M$ is the order of $\mathbb{R}$ (Lemma~\ref{sqr}). Thus the conjuncts of $\mathit{Hyp}^{\mathit{Pow}}_n(y)$ are statements about the function $x\mapsto x^y$ on $M$.
\begin{itemize}
\item[(0)] Totality: for $x\in M$, the $y$-fold product $x^y$ lies in $M$.
\item[(1)] Proximity to $1$: choose a rational $x$ with $1<x\le 2^{1/y}$; then $x\in M$ and $x^y\le2$.
\item[(2)] Strict monotonicity: for real $x>1$ and $y\ge2$ we have $x^y>x$, and $x\mapsto x^y$ is strictly increasing on $(1,\infty)$; hence for $x,x_1>1$, $x_1^y>x^y$ implies $x_1>x$.
\item[(3)] Density: let $c_1\ne c_2$ be elements of $M$ greater than $1$, say $c_1<c_2$. Taking positive real $y$-th roots, $1<c_1^{1/y}<c_2^{1/y}$, so there is a rational $q$ with $c_1^{1/y}<q<c_2^{1/y}$; then $q>1$ and $c_1<q^y<c_2$, because $x\mapsto x^y$ is strictly increasing on the positive reals. Put $x_0:=q$ and $c_0:=q^y$; both are rational, hence in $M$.
\qedhere
\end{itemize}
\end{proof}

We rely on a classical theorem from S. Lang \cite{Lang12} for the subsequent lemma.
\begin{thm}[{\cite[p.~297]{Lang12}}]\label{Langthm}
Let $K$ be a field, and $n\geq 2$ an integer. Let $a\in K$, $a\neq 0$. Assume that for all prime numbers $p$ such that $p\vert n$, $a$ is not a p-th power in $K$ (i.e., $\forall x\in K \; x^p\neq a$), and assume that if $4\vert n$ then $\forall x\in K \; a\neq -4x^4$.

Then the polynomial $X^n - a$ is irreducible in the ring of polynomials $K[X]$.
\end{thm}

\begin{lemma}\label{FieldExt}
Let $A$ be a field extension of the rational numbers $\mathbb{Q}$. Let $A(\xi)$ be the field extension of $A$ with $\xi$, where $\xi$ is a root of some non-zero polynomial $P(X)$ in $A[X]$. Suppose the degree of $P(X)$ is strictly smaller than the prime integer $p$, and suppose $A$ does not contain any root of the polynomial $X^p-2$.

Then the field $A(\xi)$ does not contain any root of the polynomial $X^p-2$.
\end{lemma}
\begin{proof}
Because the degree of $P(X)$ is strictly smaller than $p$, the degree of the extension $A(\xi)$ over $A$ is strictly smaller than $p$, as the degree of the extension equals the degree of the minimal polynomial of $\xi$, which divides $P(X)$.

Assume for contradiction that $A(\xi)$ contains a root of $X^p-2$, denoted by $\beta$. Hence, $\beta$ is algebraic over $A$. Consider the minimal polynomial $P_{\beta}$ over $A$; $P_{\beta}$ resides in $A[X]$. Because $P_{\beta}$ is minimal, it divides all polynomials in $A[X]$ sharing $\beta$ as a root; in particular, $P_\beta$ divides $X^p-2$.

By Theorem~\ref{Langthm} (applied with $n=p$ and $a=2$: the only prime dividing $p$ is $p$ itself, $2$ is not a $p$-th power in $A$ by hypothesis, and $4\nmid p$), $X^p-2$ is irreducible in $A[X]$, meaning $P_{\beta}$ and $X^p-2$ must be scalar multiples. Since $P_{\beta}$ is minimal, it is monic, rendering $P_{\beta} = X^p-2$.

By the multiplicativity of degrees of field extensions, the degree of $P_{\beta}$ must divide the degree of the extension $A(\xi)$ over $A$. Yet the degree of $P_{\beta}$ is exactly $p$, while the extension degree is strictly smaller than $p$. This contradiction ensures $A(\xi)$ does not contain a root of $X^p-2$.
\end{proof}

\begin{lemma}\label{lem:RCFlower}
Let $n\ge1$. Every proof of $\mathit{Root}_n$ in $\mathsf{RCF}$ must use an axiom whose size is greater than $2^{2^{n}-1}$. Consequently, every proof of $\mathit{Root}_n$ in $\mathsf{RCF}$ has size greater than $2^{2^{n}-1}$.
\end{lemma}
\begin{proof}
First, by Bertrand's postulate (proved by Chebyshev \cite{Che52}), since $n\ge1$ and hence $2^{2^n-1}\ge2$, there is a prime $p$ with $2^{2^{n}-1}<p<2^{2^{n}}$. Note that $p\ge3$, so $p$ is odd. We fix $p$ for the rest of the proof.

The strategy of this proof is to construct a subfield $M_p$ of $\mathbb{R}$ that satisfies $\mathsf{RCF}^-$, axiom (1) of $\mathsf{RCF}$, and the odd-degree root axioms
$$\forall a_0, ... , a_{m-1} \exists x \; x^m + a_{m-1}x^{m-1} + \ldots + a_0 = 0 $$
for all odd $m<p$, as well as $\mathit{Hyp}^{\mathit{Pow}}_n(p)$ and $p>1$, but which fails $\exists r\,\mathit{Pow}_{n}(r,p,2)$; hence $M_p\not\models\mathit{Root}_n$.

\emph{Construction of $M_p$.} Let $\mathbb{Q}_0:=\mathbb{Q}$. Given the subfield $\mathbb{Q}_j$ of $\mathbb{R}$, let $\mathbb{Q}_{j+1}$ be the subfield of $\mathbb{R}$ generated by $\mathbb{Q}_j$ together with all real roots of polynomials in $\mathbb{Q}_j[X]$ of odd degree strictly smaller than $p$, and the real square roots of all positive elements of $\mathbb{Q}_j$. This yields a chain $\mathbb{Q}_0\subseteq\mathbb{Q}_1\subseteq\mathbb{Q}_2\subseteq\cdots$ of countable subfields of $\mathbb{R}$; let $\mathbb{Q}_\omega:=\bigcup_{j}\mathbb{Q}_j$, which is a subfield of $\mathbb{R}$, being the union of a chain of fields. Let us take $M_p:=\mathbb{Q}_\omega$.

\emph{$\mathbb{Q}_\omega$ satisfies the required axioms.} Every non-negative element of $\mathbb{Q}_\omega$ lies in some $\mathbb{Q}_j$ and has its square root in $\mathbb{Q}_{j+1}\subseteq\mathbb{Q}_\omega$; hence $\mathbb{Q}_\omega$ is Euclidean and, by Lemma~\ref{sqr}, $\mathbb{Q}_\omega\models\mathsf{RCF}^-$, and axiom (1) of $\mathsf{RCF}$ holds as well (Remark~\ref{rem:sqrt}). Given a monic polynomial $P(X)\in\mathbb{Q}_\omega[X]$ of odd degree $m<p$, its coefficients all belong to some $\mathbb{Q}_j$; being of odd degree, $P$ has a real root, which by construction lies in $\mathbb{Q}_{j+1}\subseteq\mathbb{Q}_\omega$. Thus the odd-degree axioms for all odd $m<p$ hold in $\mathbb{Q}_\omega$. By Lemma~\ref{RCFmodelsHyp} (applicable since $1<p\le2^{2^n}$), $\mathbb{Q}_\omega\models\mathit{Hyp}^{\mathit{Pow}}_n(p)$.

\emph{$\mathbb{Q}_\omega$ contains no $p$-th root of $2$.} By Lemma~\ref{lem:Pow}, $\mathbb{Q}_\omega\models\mathit{Pow}_n(r,p,2)$ if and only if $r^p=2$. Since $\mathbb{Q}_\omega$ is a subfield of $\mathbb{R}$ and $p$ is odd, the equation $r^p=2$ has exactly one possible solution, namely $\beta:=2^{1/p}$. Hence it suffices to show $\beta\notin\mathbb{Q}_\omega$. First, $\beta\notin\mathbb{Q}_0=\mathbb{Q}$. Now suppose $\beta\in\mathbb{Q}_\omega$ and let $j\ge1$ be minimal with $\beta\in\mathbb{Q}_j$. Since $\mathbb{Q}_{j-1}$ is countable, $\mathbb{Q}_j$ is generated over $\mathbb{Q}_{j-1}$ by a sequence $s_0,s_1,\ldots$ of real numbers, each of which is a root of a polynomial in $\mathbb{Q}_{j-1}[X]$ of degree strictly smaller than $p$ (odd degree $<p$, or degree $2<p$). Consequently $\beta\in\mathbb{Q}_{j-1}(s_0,\ldots,s_k)$ for some $k$, which we take minimal. Let $A:=\mathbb{Q}_{j-1}(s_0,\ldots,s_{k-1})$ (so $A=\mathbb{Q}_{j-1}$ if $k=0$). By the minimality of $k$, respectively of $j$, we have $\beta\notin A$, and therefore, by the observation just made, $A$ contains no root of $X^p-2$. Now $s_k$ is a root of a polynomial of degree $<p$ with coefficients in $\mathbb{Q}_{j-1}\subseteq A$, so Lemma~\ref{FieldExt} shows that $A(s_k)$ contains no root of $X^p-2$, contradicting $\beta\in A(s_k)$.

We thus have that $M_p:=\mathbb{Q}_\omega$ satisfies the hypotheses of $\mathit{Root}_n$ at $y=p$ but not its conclusion, so $M_p\not\models\mathit{Root}_n$.

\emph{Conclusion.} Let $\pi$ be an $\mathsf{RCF}$-proof of $\mathit{Root}_n$. If all axioms of $\mathsf{RCF}$ occurring in $\pi$ were true in $M_p$, then by soundness $M_p\models\mathit{Root}_n$, a contradiction. Hence $\pi$ uses an axiom that is false in $M_p$. The axioms of $\mathsf{RCF}^-$, axiom (1), and the odd-degree axioms of degree $m<p$ are true in $M_p$, so the axiom in question is an odd-degree axiom of some degree $m\ge p$ (the model $M_p$ may also fail such axioms for some $m>p$, which only helps). By Remark~\ref{rem:sqrt}, the size of this axiom is greater than $m\ge p>2^{2^{n}-1}$. Since axioms occur as lines of proofs (Section~\ref{sec:conventions}), the size of $\pi$ is also greater than $2^{2^{n}-1}$.
\end{proof}

\begin{thm}[Speedup]\label{thm:speedupRCF}
There is a constant $\varepsilon>0$ such that $\mathsf{Tarski}$ has a $2^{2^{x^{\varepsilon}}}$-speedup over $\mathsf{RCF}$, witnessed by the sentences $\mathit{Root}_n$. More precisely: $\mathit{Root}_n$ has size polynomial in $n$, has a $\mathsf{Tarski}$-proof of size polynomial in $n$, and, for $n\ge1$, every $\mathsf{RCF}$-proof of $\mathit{Root}_n$ has size greater than $2^{2^{n}-1}$.
\end{thm}
\begin{proof}
First we note that both theories prove the sentences $\mathit{Root}_n$. For $\mathsf{Tarski}$ this is Lemma~\ref{lem:Tarskipoly}. For $\mathsf{RCF}$, observe that $\mathit{Root}_n$ is true in the field of reals: by Lemma~\ref{lem:Pow}, any real $y$ with $\mathit{Hyp}^{\mathit{Pow}}_n(y)$ and $y>1$ is a natural number, and $2$ has a real $y$-th root $r$, for which $\mathit{Pow}_n(r,y,2)$ holds. Since the models of $\mathsf{RCF}$ are exactly the real closed fields (with the order defined as above), $\mathsf{RCF}$ is a complete axiomatization of the elementary theory of the field of reals \cite{Tar51}, and hence proves $\mathit{Root}_n$.

Let $m_n$ be the minimal size of a $\mathsf{Tarski}$-proof of $\mathit{Root}_n$, which exists by Lemma~\ref{lem:Tarskipoly}. Then Lemma~\ref{lem:Tarskipoly} gives $m_n\le Cn^{c}$ for some constants $C,c\ge1$, and we have $m_n\ge|\mathit{Root}_n|\to\infty$ (by Lemma~\ref{lem:RCFlower}, each sentence occurs only finitely often in the sequence $(\mathit{Root}_n)_{n<\omega}$, and there are only finitely many sentences of any given size). Put $\varepsilon:=1/(2c)$. Then $m_n^{\varepsilon}\le C^{\varepsilon}\sqrt n\le n-1$ for all sufficiently large $n$, so that by Lemma~\ref{lem:RCFlower} every $\mathsf{RCF}$-proof of $\mathit{Root}_n$ has size greater than $2^{2^{n}-1}\ge 2^{2^{n-1}}\ge 2^{2^{m_n^{\varepsilon}}}$ for all sufficiently large $n$, which is the definition of a $2^{2^{x^{\varepsilon}}}$-speedup. As before, the choice of the proof system only affects the constant $\varepsilon$.
\end{proof}

\section{Upper Bounds for Proofs in \texorpdfstring{$\mathsf{PrA}_{\mathit{alt}}$}{PrA-alt}}\label{sec:upper}
In this section, we shall describe a modified version of Cooper's algorithm that will enable us to show that any sentence true in the standard model $(\mathbb{N};0,1,+)$ (equivalently, by soundness and Theorem~\ref{thm:triple}, any theorem of $\mathsf{PrA}_{\mathit{alt}}$) has a proof in $\mathsf{PrA}_{\mathit{alt}}$ of at most triple exponential size in the size of the proven sentence. We modify the algorithm in order to deal with the fact that we work without subtraction (in contrast with Cooper's algorithm, defined for the integers). We shall thus show that the triple exponential bound shown by Oppen \cite{Opp78} leads to a triple exponential bound for the size of the proof in the axiomatic $\mathsf{PrA}_{\mathit{alt}}$.

The quantifier elimination algorithm equivalently transforms arbitrary formulas to quantifier-free formulas in the signature expanded by $x\le y$ and the relations $x\equiv_m y$, where $m$ ranges over all standard positive naturals. For the purposes of the present paper, we use this algorithm to manipulate formulas of the original basic $\mathsf{PrA}$-signature, where instead of the extra predicate characters we are using their first-order definitions.

We now describe how, given a quantifier-free formula $F(x)$, we equivalently (over the standard model of Presburger arithmetic) transform the formula $\exists x \,F(x)$ to a quantifier-free formula $F'$ where $F'$ does not contain the variable $x$ nor any new variables. Our equivalent transformation generally follows Oppen's description of Cooper's algorithm \cite{Opp78}, but adapts it to the case where the quantifiers range over naturals instead of integers (to match the version of Presburger arithmetic that we consider in the present paper).

\textbf{Step 1.}
We equivalently transform all the atomic formulas in $F(x)$ (equalities $t=u$, inequalities $t\le u$, and modulo comparisons $t\equiv_m u$) to the form where all occurrences of $x$ are on one side of the comparison, by taking off $x$'s from both sides until one of the sides has no more $x$'s. If $x$ no longer occurs in the obtained formula at all, the procedure stops, and the equivalence with the original formula is provable in $\mathsf{PrA}^-$.

\textbf{Step 2.}
We further equivalently transform all the atomic formulas involving $x$ to the form where $x$ can be, as before, only on one of the sides, but furthermore, the number of $x$'s is the same in all of the atomic formulas having $x$'s at all. For this, we equivalently transform $t= u$ to $tk= uk$, $t\le u$ to $tk\le uk$, and $t\equiv_m u$ to $tk\equiv_{mk} uk$ for suitably chosen positive naturals $k$ (depending on the number of $x$'s in the corresponding original atomic formula). Note that here, as before, $tk$ is a shorthand for $\underbrace{t+(\ldots+(t+t)\ldots)}\limits_{\mbox{\footnotesize $k$-many $t$'s}}$.
Precisely, as in Cooper's algorithm, let $C$ be the least common multiple of the numbers $\nu_j$ of occurrences of $x$ in the atomic formulas containing $x$, and multiply the $j$-th such formula by $k:=C/\nu_j$.
Then, by applying commutativity and associativity, we put all the sides of the atomic formulas involving $x$ into the form $xC+t_i$, where $t_i$ contains no occurrences of $x$ and $C$ is the same for all the atomic formulas.

\textbf{Step 3.}
We further equivalently transform all the atomic formulas involving $x$ so that the sides involving $x$ are exactly the same in all the formulas and are of the form $xC+\Lambda$. For this, we collect all the $t_i$'s and, for each atomic formula involving $x$, we add all the $t_i$'s other than the one corresponding to that particular formula to both sides.

Thus, we obtained an equivalent transformation of $F(x)$ to the form $F''(Cx+\Lambda)$, for a certain $F''(y)$.

\textbf{Step 4.}
We equivalently transform $\exists x\, F(x)$ to $\exists y\, F'''(y)$, where $y$ can occur only as a whole one side of an atomic formula, and furthermore no more than once in each atomic formula. For this, we set $F'''(y)$ to be $\Lambda\le y \land y\equiv_C \Lambda \land F''(y)$. (Note that $C\ge1$.)

\textbf{Step 5.}
We equivalently transform $\exists y\, F'''(y)$ to the form
\[F'\colon \bigvee_{t\in T}\bigvee\limits_{0\le r\le D} F'''(t+r),\]
where $D$ is the least common multiple of $m$'s of the predicates $\equiv_m$ used in $F'''(y)$ and $T$ is the set consisting of constant $0$ and all $t$'s from atomic formulas from $F'''(y)$ of the forms $t\le y$, $y\le t$, $t=y$, $y=t$.

The correctness proof of Step 5 in $\mathsf{PrA}_{\mathit{alt}}$ requires division with remainder modulo the composite number $D$, since Step 2 multiplies moduli, whereas the axioms of $\mathsf{PrA}_{\mathit{alt}}$ only provide it for primes. This is the content of the following lemma.

\begin{lemma}[Division with remainder modulo $D$]\label{lem:modD}
For every positive natural number $D$, the sentence
\[\forall x\,\bigvee_{0\le r<D}\exists z\,(Dz+r=x)\]
has a proof in $\mathsf{PrA}_{\mathit{alt}}$ of size polynomial in $D$; in particular, $\mathsf{PrA}_{\mathit{alt}}$ proves $\forall x\,\bigvee_{0\le r<D}x\equiv_D r$ by a proof of size polynomial in $D$. Moreover, the uniqueness statement
\[\forall x\bigwedge_{0\le r<r'<D}\neg\,(x\equiv_D r\wedge x\equiv_D r')\]
has a proof in $\mathsf{PrA}^-$ of size polynomial in $D$.
\end{lemma}
\begin{proof}
We first note that, for every $m\ge1$ and every $s<m$, $\mathsf{PrA}^-$ proves $x\equiv_m s\rightarrow\exists z\,(mz+s=x)$ by a proof of size polynomial in $m$. In particular, the axiom of $\mathsf{PrA}_{\mathit{alt}}$ for a prime $q$ yields $\forall x\bigvee_{s<q}\exists z\,(qz+s=x)$ by a proof of size polynomial in $q$.

Write $D=q_1q_2\cdots q_k$ with all $q_i$ prime (repetitions allowed), and let $D_i:=q_1\cdots q_i$, so $D_0=1$ and $D_k=D$. We show by external induction on $i\le k$ that $\mathsf{PrA}_{\mathit{alt}}$ proves $\forall x\bigvee_{r<D_i}\exists z\,(D_iz+r=x)$ by a proof of size polynomial in $D$. For $i=0$ the sentence is $\forall x\,\exists z\,(z=x)$. For the induction step, assume $D_iz+r=x$ with $r<D_i$. The statement for the prime $q_{i+1}$ applied to $z$ gives $z'$ and $s<q_{i+1}$ with $q_{i+1}z'+s=z$. Then
\[x=D_i\cdot(q_{i+1}z'+s)+r=D_{i+1}z'+(D_is+r),\]
where the last identity is an instance of the associativity and commutativity of $+$ (recall that multiplication by a numeral is an iterated sum, so both sides are sums of $z'$'s and $1$'s), and $D_is+r<D_{i+1}$ is a true comparison of numerals. Thus $x$ satisfies one of the $D_{i+1}$ disjuncts $\exists z'\,(D_{i+1}z'+r'=x)$ with $r'=D_is+r<D_{i+1}$. Each disjunct has size $O(D)$, the case distinction has $D_i\cdot q_{i+1}\le D$ cases, and each case has a proof of size polynomial in $D$; hence the whole proof has size polynomial in $D$. The statement $\forall x\bigvee_{r<D}x\equiv_D r$ follows immediately, since $Dz+r=x$ is the first disjunct of the shorthand $x\equiv_D r$.

The uniqueness follows by considering $Dz+r=x=Dz'+r'$ with $r,r'<D$ and using cancellation and properties of the order; the proofs we obtain are of size polynomial in $D$.
\end{proof}

\begin{remark}\label{rem:modD}
The remaining elementary properties of the relations $\equiv_m$ that we need---that $\equiv_m$ is an equivalence relation, that $x\equiv_m y$ implies $x+z\equiv_m y+z$, and that $x\equiv_D y$ implies $x\equiv_m y$ whenever $m\mid D$---are provable in $\mathsf{PrA}^-$ by proofs of size polynomial in the moduli involved, by the same kind of manipulation of iterated sums.
\end{remark}

\begin{thm} \label{equivalent_transform}
Let $F'$ be the quantifier-free formula described by our algorithm for $\exists x\, F(x)$. Then the proof in $\mathsf{PrA}_{\mathit{alt}}$ of the equivalence between the formulas $\exists x F(x)$ and $F'$ is of polynomial size in the size of the end formula $\exists x F(x)\leftrightarrow F'$.
\end{thm}
\begin{proof}[Proof sketch]
    The proof of the equivalence $\exists x \,F(x)\mathrel{\leftrightarrow} \exists y\, F'''(y)$ (i.e., steps 1--4 of the algorithm) is formalizable in a fairly straightforward manner. Indeed, steps 1--4 consist of local rewritings: each replaces a formula
(in step~4, the formula $\exists x\,F''(Cx+\Lambda)$) by a $\mathsf{PrA}^-$-provably
equivalent one, with a proof of size polynomial in the sizes of the two
formulas. Composing these equivalences by replacement of equivalents,
$\mathsf{PrA}^-$ proves $\exists x\,F(x)\leftrightarrow\exists y\,F'''(y)$ by a
proof of size polynomial in $|\exists x\,F(x) \leftrightarrow \exists y\,F'''(y)|$ (hence in
$|\exists x\,F(x)\leftrightarrow F'|$, since $F'$ contains substitution instances of $F'''$); if $x$
no longer occurs after step~1, the procedure stops there, and the remaining equivalence
(dropping the vacuous quantifier) is pure logic. It thus remains to prove $\exists y\,F'''(y)\leftrightarrow F'$.

    The direction from right to left is trivial: each disjunct $F'''(t+r)$ of $F'$ directly witnesses $\exists y\,F'''(y)$. For the direction from left to right, we formalize within $\mathsf{PrA}_{\mathit{alt}}$ the following argument. We consider $y$ such that $F'''(y)$ holds and claim that $F'$ holds. If $y$ is equal to one of $t\in T$, then we are already done (take $r=0$), so further we assume that $y$ is distinct from all $t\in T$. We observe that the following formula has a proof in $\mathsf{PrA}^-$ of size polynomial in the size of $T$:
    \[\bigwedge_{u\in T} y\ne u\;\rightarrow\;\bigvee_{t\in T}\Big( t<y\land \bigwedge\limits_{u\in T} (u\le y\to u\le t)\land \bigwedge\limits_{u\in T} (t<u\to y<u)\Big).\]
    That is, if $y\notin T$, then there is a maximal element $t$ among $\{t\in T\mid t\le y\}$ (this set is provably non-empty, since $0\in T$ and $0\le y$), this maximal element is distinct from $y$ itself, and every $u\in T$ above $t$ is above $y$. The proof proceeds by external induction on the number of terms in $T$, producing a sequence of proofs of polynomially growing size: adding a term $u_0$ to $T$, one distinguishes the cases $u_0\le y$ and $y<u_0$ and, in the former case, compares $u_0$ with the maximal element obtained for the smaller set.
    Continuing further, we proceed by proof by cases, assuming $t\in T$ to be fixed to satisfy the respective disjunct just above. Since $t<y$, there is $w$ with $t+w+1=y$. By Lemma~\ref{lem:modD} applied to $w+1$, there are $z$ and $s<D$ with $Dz+s=w+1$. Put $r:=s$ if $s\ne0$, and $r:=D$ if $s=0$ (in the latter case $z\ne0$, since $w+1\ne0$, so $D\le w+1$). In both cases $1\le r\le D$, $t<t+r\le y$, and $t+r\equiv_D y$.

    To finish the proof, we will show for each atomic formula $A(y)$ in $F'''(y)$ that $A(y)\mathrel{\leftrightarrow} A(t+r)$. Recall that, by Step 4, $y$ occurs in $A(y)$ at most once, as a whole side of the atomic formula. Since $t<t+r\le y$ and $t$ is maximal in $T$ below $y$, the element $t+r$
compares with all $u\in T$ in exactly the same way as $y$: for each $u\in T$ we have $y=u\mathrel{\leftrightarrow} t+r=u$ (both sides are false, since $y\notin T$, and $t+r=u$ would give $u\le y$, hence $u\le t<t+r$), $y\le u\mathrel{\leftrightarrow} t+r\le u$, and $u\le y\mathrel{\leftrightarrow} u\le t+r$ (if $u\le y$ then $u\le t<t+r\le y$, so both sides of both equivalences agree; if $\neg\, u\le y$, then $y<u$ by axiom 4, hence $t+r\le y<u$, and again both sides agree). And for each modular atom $y\equiv_m v$ (or $v\equiv_m y$) of $F'''$, from
$t+r\equiv_D y$ and $m\mid D$ we get $t+r\equiv_m y$, hence
$t+r\equiv_m v$ by transitivity (Remark~\ref{rem:modD}). By replacement of equivalents, these equivalences for the atomic formulas
yield $F'''(y)\leftrightarrow F'''(t+r)$, which is the disjunct of $F'$
corresponding to $t$ and $r$. The above argument scheme, instantiated with the terms of $T$, the numerals
up to $D$ and the atomic formulas of $F'''$, yields a proof of size
polynomial in $|T|+D+|F'''|$, hence in the size of
$\exists x\,F(x)\leftrightarrow F'$.
\end{proof}

\begin{lemma} \label{quantifier_free_proofs} Every quantifier-free sentence $F$ in the language of $\mathsf{PrA}$ expanded by $x\le y$ and $x\equiv_m y$ that is true in the standard model $(\mathbb{N},0,1,+)$ has a proof in $\mathsf{PrA}_{\mathit{alt}}$ of a length polynomial in the length of $F$.\end{lemma}
\begin{proof}[Proof sketch]
    Clearly, it suffices to prove the lemma for true literals (of the extended signature). The proof for the case of literals is routine. First, with a polynomial proof, we equivalently transform the formula of interest to a form where both sides are in the form of a numeral. Then we establish polynomial $\mathsf{PrA}_{\mathit{alt}}$-proofs of the properties giving inductive definitions of atomic formulas:
    \begin{enumerate}
        \item $0 = 0$;
        \item $\lnot 0=x+1$;
        \item $\lnot x+1=0$;
        \item $x+1 = y+1\mathrel{\leftrightarrow}x= y$;
        \item $0\le x$;
        \item $\lnot x+1\le 0$;
        \item $x+1\le y+1\mathrel{\leftrightarrow}x\le y$;
        \item $0\equiv_m 0$;
        \item $\lnot 0\equiv_m k$, for $1\le k<m$;
        \item $\lnot k\equiv_m 0$, for $1\le k<m$;
        \item $0\equiv_m x+m\mathrel{\leftrightarrow} 0\equiv_m x$;
        \item $x+1\equiv_m y+1\mathrel{\leftrightarrow} x\equiv_m y$.
    \end{enumerate}
    Using them, it is straightforward to show that the true closed literals have polynomial proofs.
\end{proof}

\begin{thm}\label{thm:triple}
For every sentence $F$ in the language of $\mathsf{PrA}$ that is true in the standard model $(\mathbb{N},0,1,+)$, there is always a proof of this sentence in $\mathsf{PrA}_{\mathit{alt}}$ that uses at most $2^{2^{2^{P(|F|)}}}$ characters, where $P(x)$ is some polynomial.
\end{thm}
\begin{proof}[Proof sketch]
 First, let us remark that our algorithm transforms sentences in a similar way as the one analyzed by Oppen \cite{Opp78}. More precisely, the argument has the following structure.

\emph{Innermost-first elimination.} We first pass from $F$ to a sentence
$F_0$: we rename the bound variables to $x_1,\dots,x_V$, $V\le|F|$, and
replace every subformula $\forall x\,A$ by $\neg\exists x\,\neg A$. Thus
$F_0$ has no universal quantifiers, the same atomic formulas up to renaming of variables and the same
number $K\le|F|$ of quantifiers as $F$, and at most $2K$ more connectives.
For a term, let $\alpha_v$ ($1\le v\le V$) be the number of occurrences of
$x_v$ in it and $\alpha_0$ the number of occurrences of the constant symbols
$0$ and $1$ in it. We eliminate quantifiers in place, innermost first: an
innermost subformula $\exists x\,G$ with $G$ quantifier-free is replaced by
the output $G'$ of the algorithm. Indeed Steps 1--5 apply verbatim to
quantifier-free formulas of arbitrary Boolean structure. Each round removes
one quantifier, so after $K$ rounds we reach a quantifier-free sentence
$F^{*}$; let $F_0,F_1,\dots,F_K=F^{*}$ be the sentences produced.

\emph{Replacement of equivalents.} Each round replaces an innermost subformula $\exists x\,G$ of $F_k$ by the
provably equivalent quantifier-free formula $G'$ supplied by
Theorem~\ref{equivalent_transform}, which yields $F_{k+1}$. The equivalence $F_k\leftrightarrow F_{k+1}$ is then derived by the replacement-of-equivalents rule, by a derivation of size polynomial in $|F_k|+|F_{k+1}|$. Thus each round costs a proof of size polynomial in $|F_k|+|F_{k+1}|$.

\emph{Size counting.} For a formula let $a$ be its number of atomic formulas; $b$ its number of connectives; $s$ the largest value among its moduli and the numbers
$\alpha_0,\dots,\alpha_V$ of its terms; and $c$ the number of distinct
values taken by its coefficients $\alpha_1,\dots,\alpha_V$ (zeros
included) and moduli, the $\alpha_0$ of its terms not being counted. 

For each formula $F_k$ from our equivalent transformation process let $a_k,b_k,c_k$, and $s_k$ be the $a,b,c,$ and $s$ corresponding to $F_k$. In order to establish the desired triple exponential upper bound on the sizes of $F_k$, the strategy will be to give recursive upper bounds for $a_{k+1},b_{k+1},c_{k+1},$ and $s_{k+1}$ in terms of $a_{k},b_{k},c_{k},$ and $s_{k}$. The crucial point will be that the bound on $c_{k+1}$ is a polynomial in terms of $|F|$ and $c_k$, thus leading to a double exponential in $|F|$ upper bound for all $c_k$. Then the bounds for $a_{k+1},b_{k+1}$, and $s_{k+1}$ will be bounded by an expression of the form  $P^{Q(c_k)}(a_k,b_k,s_k)$, for some polynomials $P$ and $Q$, which combined with a double exponential bound for $c_k$, leads to triple exponential bounds for $a_k,b_k,$ and $s_k$, which in turn gives the triple exponential upper bounds for $|F_k|$.

Initially, for $F$ itself, all four parameters are at most $|F|$, and
$K,V\le|F|$. The parameters $a_0,b_0,c_0,s_0$ of $F_0$ therefore satisfy $a_0,c_0,s_0\le|F|$ and
$b_0\le|F|+2K\le3|F|$. The aim is to
bound the parameters of $F_{k+1}$ in terms of the parameters of $F_k$; iterating these
bounds will then give bounds for all $F_k$ in terms of the parameters of
$F_0$, and hence of $|F|$. These in turn will give a bound on the size
$|F_k|$ in terms of $|F|$. 

Now let $a,b,c,s$ be the parameters of $F_k$, and let us follow one round,
applied to an innermost $\exists x\,G$, through the steps, to see how these parameters
change; the parameters of $G$ itself are at most those of $F_k$. Steps 1--3 leave $a$ and $b$ unchanged. Step 1 turns the coefficients of $x$ into differences of two coefficients, so after Step~1 the coefficients and moduli take at most $c':=c+c^2\le2c^2$ distinct values, each $\le s$; since a least common multiple is at most the product of the distinct values involved, $C\le s^{2c^2}$. Step 2 multiplies each atomic formula containing $x$ by $C/\nu_j\le C$, so
all numbers become $\le Cs$; since $C/\nu_j$ is determined by $\nu_j$,
there are at most $c'$ distinct multipliers, so the coefficients and
moduli take at most $c'':=c'+c'^2+1\le3c'^2$ distinct values. Step 3 adds sums of at most $a$ such numbers, so the numbers become
$\le(a+1)Cs$, and the new coefficients have the form $q+(\Lambda_v-q')$,
where $\Lambda_v$ ($1\le v\le V$) is the coefficient of $x_v$ in
$\Lambda=\sum_it_i$ and $q,q'$ are old coefficients or $0$; each new value is
thus determined by a triple $(v,q,q')$, and there are at most
$V(c''+1)^2$ such triples. 

As for the moduli, let $L$ be the least common
multiple of the moduli of $G$ (the moduli before Step~2). After Step~2
each modulus $m_j$ of an atomic formula containing $x$ has become
$m_jC/\nu_j$, which divides $m_jC$ and hence $LC$; the moduli of the
other atomic formulas divide $L$; Step~3 leaves all moduli unchanged; and
Step~4 adds the single modulus $C$, which divides $LC$. Thus every
modulus of $F'''$ divides $LC$, and so does their least common multiple
$D$. Since the moduli of $G$ take at most $c$ distinct values, each
$\le s$, we have $L\le s^{c}$, hence $D\le LC\le s^{c}\cdot s^{2c^2}\le
s^{3c^2}$. Steps~4 and~5 add no further values: every $t \in T$ is $0$, the term $\Lambda$
of Step~4 (with coefficients $\Lambda_v$, corresponding to the triples $(v,0,0)$) or a side
produced in Step~3 (the triples $(v,q,q')$), so the coefficients of $t+r$, which are those of $t$ (the numeral $r$ affects
$\alpha_0$ only), have already been counted; and the only new modulus is $C$,
which is already accounted for in $c''$. Hence the moduli of $F'$ are among the at most $c''$
values present after Step~2, and the coefficients and moduli of $F'$ take
at most $V(c''+1)^2+c''$ distinct values. 

Finally $|T|\le a+2$, so $F'$ is a disjunction of at
most $(a+2)(D+1)\le2(a+2)s^{3c^2}$ instances of $F'''$, each with at most $a+2$
atomic formulas and $b+2$ connectives, whose numbers are
$\le(a+1)Cs+D\le(a+2)s^{3c^2+1}$. Since $F_{k+1}$ is $F_k$ with $\exists x\,G$ replaced by $F'$, we have that  every
atomic formula, connective, coefficient and modulus of $F_{k+1}$ occurs
either in $F'$ or in the part of $F_k$ outside $G$, and that part has at
most $a$ atomic formulas, $b$ connectives and $c$ distinct values, and
numbers at most $s$. Adding these to the bounds for $F'$, and writing $a_k,b_k,c_k,s_k$ for the
parameters of $F_k$, we hence obtain for the parameters of $F_{k+1}$:
\begin{gather*}
c_{k+1}\le c_k +V(c''+1)^2+c''\le2^{9}Vc_k^{8},\qquad s_{k+1}\le(a_k+2)\,s_k^{3c_k^2+1},\\
a_{k+1}\le a_k+2(a_k+2)^2s_k^{3c_k^2},\qquad b_{k+1}\le b_k+2(a_k+2)(b_k+3)\,s_k^{3c_k^2}.
\end{gather*}
Solving the recurrences: since $c_0\le|F|$ and $2^9V\le2^{10}|F|$, induction gives $c_k\le(2^{10}|F|)^{1+8+\dots+8^{k}}=(2^{10}|F|)^{(8^{k+1}-1)/7}\le(2^{10}|F|)^{2\cdot8^{k}}\le\Gamma:=(2^{10}|F|)^{2\cdot8^{|F|}}\le2^{2^{4|F|+5}}$ for all $k\le K$; and with $u_k:=\max(a_k+2,\,s_k,\,b_k+3)\ge2$ the other three recurrences give $u_{k+1}\le4u_k^{3\Gamma^2+2}\le u_k^{4\Gamma^2}$, so since $u_0\le3|F|+3\le4\cdot2^{|F|}=2^{|F|+2}$, we have $u_k\le u_0^{(4\Gamma^2)^{|F|}}\le2^{(|F|+2)(4\Gamma^2)^{|F|}}\le2^{2^{2^{5|F|+8}}}$. 

It remains to convert these bounds on the parameters into a bound on the
size $|F_k|$. A term has at most $(V+1)s_k+1$ leaves and hence $O((|F|+1)s_k)$ elements; writing out an atomic formula in full, i.e.\ expanding the shorthands $\le$ and $\equiv_m$ (which duplicates its two terms and adds two $m$-fold sums with $m\le s_k$) and writing each variable name with its $O(\log|F|)$ symbols, therefore costs at most $\kappa(|F|+2)^2s_k$ symbols for some fixed constant $\kappa$. Counting the $b_k$ connectives, the at most $|F|$ quantifiers and the parentheses, we get $|F_k|\le\kappa(|F|+2)^2\,(a_ks_k+b_k+|F|)$. Since $a_ks_k+b_k+|F|\le3\cdot2^{|F|}u_k^2$ and $u_k\le2^{2^{2^{5|F|+8}}}$, this gives $|F_k|\le2^{2^{2^{5|F|+10}}}$ for all $k\le K$.

\emph{Composition.} The equivalence $F\leftrightarrow F_0$ is provable in pure logic, by renaming
of bound variables and the equivalences
$\forall x\,A\leftrightarrow\neg\exists x\,\neg A$, with a proof of size
polynomial in $|F|$. Each of the $K\le|F|$ equivalences
$F_k\leftrightarrow F_{k+1}$ has, by Theorem~\ref{equivalent_transform} and
replacement of equivalents, a proof of size polynomial in
$|F_k|+|F_{k+1}|\le2N$, where $N:=2^{2^{2^{5|F|+10}}}$. Chaining these $K+1$
equivalences gives a $\mathsf{PrA}_{\mathit{alt}}$-proof of
$F\leftrightarrow F^{*}$ of size polynomial in $N$. Since $\mathsf{PrA}_{\mathit{alt}}$ is sound for the standard model, the
sentence $F\leftrightarrow F^{*}$ just proved is true there; as $F$ is true,
so is $F^{*}$. Thus $F^*$ is a true quantifier-free sentence in the expanded signature, and so there is a polynomial-length proof of $F^*$ by Lemma \ref{quantifier_free_proofs}. Combining the proofs with modus ponens gives the required proof of $F$. The total size is therefore at most $\mathrm{poly}(N)\le N^{d}\le2^{2^{2^{5|F|+11}}}$ for a constant $d$ and all large $|F|$; so $P(x)=p_1x+p_0$ works for suitable constants $p_1$, $p_0$.
\end{proof}

\begin{remark}
In particular, Theorem~\ref{thm:triple} shows that every sentence true in $(\mathbb{N};0,1,+)$ is provable in $\mathsf{PrA}_{\mathit{alt}}$, which re-proves the completeness of $\mathsf{PrA}_{\mathit{alt}}$ (cf.\ \cite{Gra97,Vis11}); together with soundness, this shows that the theorems of $\mathsf{PrA}_{\mathit{alt}}$ are exactly the sentences true in the standard model, so that the bound of Theorem~\ref{thm:triple} applies to all theorems of $\mathsf{PrA}_{\mathit{alt}}$.
\end{remark}

\begin{remark} By a result of Berman \cite{Berman1980} Presburger arithmetic is complete for the problems decidable on alternating Turing machines running double-exponential in the size of the input time and having linearly many in the size of the input alternations. Thus, a reduction of the triple-exponential upper bound from Theorem~\ref{thm:triple} to a double-exponential would lead to coincidence of problems decided by the said class of alternating Turing machines and non-deterministic Turing machines running in double-exponential time. The latter would be rather implausible from the perspective of computational complexity theory.\end{remark}

\begin{corollary}[Upper bound on the speedup]\label{cor:upper}
The speedup of $\mathsf{PrA}$ over $\mathsf{PrA}_{\mathit{alt}}$ is at most triple exponential: there is a polynomial $P$ such that every sentence $F$ that has a $\mathsf{PrA}$-proof of size $m$ has a $\mathsf{PrA}_{\mathit{alt}}$-proof of size at most $2^{2^{2^{P(m)}}}$. Consequently, $\mathsf{PrA}$ does not have an $f$-speedup over $\mathsf{PrA}_{\mathit{alt}}$ for any function $f$ with $f(m)>2^{2^{2^{P(m)}}}$ for all large $m$.
\end{corollary}
\begin{proof}
If $F$ has a $\mathsf{PrA}$-proof of size $m$, then $F$ is true in $(\mathbb{N};0,1,+)$ by soundness, and $|F|\le m$. By Theorem~\ref{thm:triple}, with $P$ non-decreasing, $F$ has a $\mathsf{PrA}_{\mathit{alt}}$-proof of size at most $2^{2^{2^{P(|F|)}}}\le 2^{2^{2^{P(m)}}}$. The second statement follows directly from the definition of an $f$-speedup in Section~\ref{sec:conventions}.
\end{proof}

\section{Open problems}
We conclude with some questions left open by the present work.
\begin{enumerate}
\item \emph{Upper bounds for $\mathsf{RCF}$.} We conjecture that, analogously to Theorem~\ref{thm:triple}, every sentence true in all real closed fields has an $\mathsf{RCF}$-proof whose size is bounded by a fixed finite tower of exponentials in the size of the sentence, so that the speedup of $\mathsf{Tarski}$ over $\mathsf{RCF}$ would be bounded by an elementary function. We expect that this might be possible to achieve via a formalization of quantifier elimination for real closed fields via cylindrical algebraic decompositions \cite{collins1975quantifier,basu2006algorithms}.
\item \emph{Other theories.} The method of Sections~\ref{sec:PrA} and~\ref{sec:RCF}---short iterated definitions in the style of Fischer and Rabin, combined with models separating a single long axiom from its predecessors---should apply to other complete theories with quantifier elimination that admit both a schematic and a ``degree-by-degree'' axiomatization. It would be interesting to determine the exact magnitude of the speedups in such cases.
\end{enumerate}

\section{Acknowledgments}
We are grateful to Leszek Ko{\l}odziejczyk, with whom the second author had a discussion about the typical rates of speedup for pairs of first-order theories proving the same theorems within a particular class of formulas. The motivation for studying situations where the rate of speedup is of the order of a tower of exponentials of finite height largely originates from that conversation.

\section{AI use disclosure}
AI tools (\textit{Claude Fable 5, 5.1 and 5.2}) were used to check the mathematical content of this work for errors and to assist with the editing process. The following are the substantive contributions of AI to the mathematical content of the paper. Firstly, Lemma~\ref{sqr} was suggested, together with its proof, by an AI assistant to resolve a gap in a draft of the paper. Secondly, AI was used to produce a draft of the proof of Theorem~\ref{thm:triple}, thus helping us to adapt the argument of Oppen~\cite{Opp78} to the particular definitions of the present paper.

\bibliographystyle{plain}
\bibliography{refs}

\end{document}